\documentclass{amsart}

\usepackage{amsmath,amssymb,amsfonts,amstext,amsthm,bm}
\usepackage{url}
\usepackage{graphicx}
\usepackage{hyperref}
\usepackage{tikz}
\usepackage[all,cmtip]{xy} 
\usepackage{multicol}
\usepackage{verbatim}
\input xy
\xyoption{all}
\usepackage[all]{xy}
\usepackage{xcolor}
\usepackage{enumitem}

\newcommand{\trop}{\operatorname{trop}}

\newtheorem{coro}{Corollary}[section]

\newtheorem{dfn}[coro]{Definition}

\newtheorem{lemma}[coro]{Lemma}

\newtheorem{prop}[coro]{Proposition}

\newtheorem{rem}[coro]{Remark}

\newtheorem{thm}[coro]{Theorem}

\newtheorem{ex}[coro]{Example}

\let \sub=\subset

\let \al=\alpha
\let \pr=\prime
\let \la=\lambda

\let \sub=\subset
\let \mb=\mathbb
\let \mc=\mathcal
\let \wt=\widetilde

\let \ra=\rightarrow

\usepackage{faktor}

\title{Exploring tropical differential equations
}
\date{\today}

\author{Ethan Cotterill
\and Cristhian Garay L\'opez
\and Johana Luviano
}
\thanks{The second author was supported by CONACYT through Project 299261. The third author was supported by PAPIIT IN108320}
 
\begin{document}

%tropical differential algebraic geometry, power series solutions, Newton polygons, initial forms.}
%\thanks{The first author was supported by CONACYT through Project 299261.}
\maketitle 
\begin{abstract}
 
  The purpose of this paper is fourfold. The first is to develop the theory of tropical differential algebraic geometry from scratch; the second is to present the tropical fundamental theorem  for {\it  differential algebraic geometry}, and show how it may be used to extract combinatorial information about the set of power series solutions to a given system of differential equations, both in the archimedean (complex analytic) and in the non-archimedean (e.g., $p$-adic) settings. A third and subsidiary aim is to show how tropical differential algebraic geometry is a natural application of semiring theory, and in so doing, contribute to the valuative study of differential algebraic geometry. 
  %Finally, {\color{red} we use this formalism to extend the fundamental theorem of partial differential algebraic geometry to the quotient differential field of the ring of formal power series in several variables, which shows that }the methods we have used in formulating and proving the extension of the fundamental theorem reveal new examples of non-Krull valuations that merit further study in their own right.
We use this formalism to extend the fundamental theorem of partial differential algebraic geometry to differential fraction field of the ring of formal power series in arbitrarily (finitely) many variables; in doing so we produce new examples of non-Krull valuations that merit further study in their own right.
\end{abstract}

\section{Introduction}

Solving systems of  differential equations is in general a difficult task. Indeed, the results of  \cite{DL84} imply that computing formal power series solutions to systems of algebraic partial differential equations with (complex) formal power series coefficients is algorithmically impossible. 

On the other hand, the rapidly-evolving area of {\it tropical} differential algebraic geometry, introduced in \cite{grigoriev2017tropical}, provides an algebraic framework for recording precise combinatorial information about the set of formal power series solutions to systems of differential equations. The foundation of this tropical theory is an analogue of the {\it tropical fundamental theorem} (itself a tropical analogue of the classical Nullstellensatz) adapted to the differential-algebraic context.

A {\it tropical fundamental theorem} in its basic form refers to a correspondence between tropicalizations of geometric objects on one side and tropicalizations of algebraic objects on the other; and constructing such correspondences is the primary impetus for the application of tropical methods in commutative algebra.
% is to construct correspondences known as {\it fundamental theorems}, 
By {\it tropicalization} we mean schemes $(R,{\rm S},v)$ in the sense of \cite{GG,KM19}, %(see also \cite{KM19}), 
in which $R$ is a commutative ring (possibly endowed with some extra structure), ${\rm S}$ is a commutative (additively idempotent) semiring, and $v:R\rightarrow {\rm S}$ is a map that satisfies the usual properties of a seminorm. We are typically interested in the set of algebraic solutions $\text{Sol}(\Sigma)$ of a
system of algebraic equations $\Sigma\subset R^{\pr}$ for some  $R$-algebra $R^{\pr}$. Whenever the {\it tropicalization} $v(\text{Sol}(\Sigma))$ of $\text{Sol}(\Sigma)$ coincides with the solution set $\text{Sol}(v(\Sigma))$ of the {\it tropicalization} $v(\Sigma)$ of the original system, we say that a  tropical fundamental theorem holds in this context.

In the differential-algebraic setting, the first tropical fundamental theorem proved was for systems of ordinary differential algebraic  equations over rings of formal power series; this is the main result of \cite{AGT16}. In this case, the valued ring that underlies the tropicalization scheme is $R=K[\![t]\!]$, where 
$K$ is an uncountable, algebraically closed field of characteristic zero; the $R$-algebra $R^{\pr}$ is equal to the ring of ordinary differential polynomials $K[\![t]\!] \{x_1,\ldots,x_n\}$; and the $\text{Sol}(\Sigma)$ are power series solution sets to systems of ordinary differential equations $\Sigma$. This tropical fundamental theorem was subsequently generalized to systems of partial differential equations in \cite{FGH20}. 

There is a third characterization of tropical varieties in terms of initial ideals; see \cite[Thm 3.2.5]{MS}. In \cite[Lem 2.6]{HG19}, its analogue was stated for systems of ordinary differential equations; and in \cite[Thm 45]{FGH20+}, the latter result was generalized to the partial differential setting. 

This fundamental theorem  gives three equivalent characterizations of the {\it tropical DA varieties}; here DA stands for {\it differential algebraic}. It is the most important result in the area to date, as it opens the door to applying tropical methods to a broad array of problems in differential algebraic geometry. 

%\medskip
In this note, we re-frame tropical differential algebraic geometry in general, and the tropical fundamental theorem in particular, within a theory of {\it non-classical} non-Archimedean seminorms, in which there is no total order on   the target idempotent semiring. We single out a number of non-classical non-Archimedean  seminorms that may be worthy of further study. 

Our work also represents an important first step in the systematic valuative study of support and tropicalization maps; and we discuss potential applications of the fundamental theorem to algebraic differential equations, in both the archimedean and non-archimedean settings.

The main protagonists here are certain {\it vertex polynomials}, whose (finite) supports determine the Newton polyhedra of (infinite) power series, and the  semiring $V\mb{B}[T]$ they determine. $V\mb{B}[T]$ is the target of our non-classical valuation on $K[\![T]\!]$, and it has many useful properties. Crucially, we show that it is {\it integral}, which enables us to extend our valuative theory from $K[\![T]\!]$ to its quotient field $K(\!(T)\!)$; to naturally distinguish a {\it ring of integers} $K[\![T]\!]\sub K(\!(T)\!)^{\circ} \sub K(\!(T)\!)$ analogous to the  valuation ring of a classical valuation; and to formulate geometric translation maps that underlie a robust theory of {\it initial forms} associated to fixed choices of weight vectors.

\subsection{Roadmap}
The remainder of the paper following this introduction is organized as follows. Section~\ref{Section_DifEqs} is an introduction to differential algebraic geometry over (formal) power series semirings. The differential structures of semirings of differential polynomials are codified by natural $\mb{N}^m$-actions, where $m$ is the number of derivatives. These are slightly more elegant to formulate within the framework of {\it semirings of twisted polynomials} that we introduce in section~\ref{tp}.

%\medskip
In section~\ref{Tropical}, we introduce tropical differential equations and their associated tropical power series {\it solutions}, which are ``corner loci" of vertex polynomials; 
%certain {\it vertex polynomials} extracted from the Newton polyhedra of power series; 
see Definition~\ref{Def_trop_vanishing} for a formal statement. In Example~\ref{Ex_Solution} we illustrate explicitly how these two concepts work together.   

%\medskip
In section~\ref{section_Idempotent_Semiring}, we study %the idempotent semiring 
$V\mb{B}[T]$,
%of vertex polynomials, 
which we construct as a quotient of boolean power series by a semiring congruence. Lemma~\ref{New_congruence} establishes that our construction exactly reproduces the vertex set semiring introduced in \cite{FGH20}. %{\color{red} Proposition~\ref{prop_reformulation_solution} shows that the set of tropical power series solutions to a single tropical differential polynomial is indeed its corner locus, after adapting the traditional definition of corner locus of tropical polynomials to the differential setting.}  
Proposition~\ref{prop_reformulation_solution} establishes that the set of tropical power series solutions to a single tropical differential polynomial is indeed its corner locus, provided we appropriately adapt the traditional definition of corner locus of tropical polynomials to the differential setting.
In subsection~\ref{SubSection_Order}, we introduce a refinement of the  order relation on vertex polynomials, based upon the concept of {\it relevancy}; see Definition~\ref{dfn_relevant}.
%section \ref{SubSection_Order} we refine its order structure and we  introduce the important notion of relevancy. 
This refinement allows us to recover some useful features of Krull valuations for this case. Subsection~\ref{alternative_characterizations}, on the other hand, is a more detailed study of vertex polynomials through the lens of (Newton) polyhedra. Propositions~\ref{Prop_HProperty}, \ref{finite_generation_A-tilde}, and \ref{char_of_prod} establish structural results for vertex sets, and their behavior under unions and Minkowski sums. We  examine in subsection~\ref{Sub_Section_Relationship} how $V\mb{B}[T]$ relates to the semiring $\mathcal{P}_{\mathbb{Z}^m}$ of lattice polytopes in $\mathbb{Z}^m$ studied in \cite{KM19}. 

%\medskip
In section~\ref{Section_QVIF}, we introduce {\it non-Archimedean seminorms} for algebras defined over a field $K$ with targets in semirings, {\it support series} over the boolean semifield $\mb{B}$ for $K$-formal power series, and {\it support vectors} for solutions to systems of tropical algebraic differential equations. A key technical fact, established in Theorem~\ref{Longue_prop}, 
%whose proof we defer to Appendix~\ref{Proof_of_Longue_prop}, 
is that the support maps $\text{sp}: K[\![T]\!] \ra \mb{B}[\![T]\!]$ are norms that commute with the respective differentials of $K[\![T]\!]$ and $\mb{B}[\![T]\!]$. We further show in Theorem~\ref{VBT_theorem} that composing the support maps with the quotient projection $\mb{B}[\![T]\!] \ra V\mb{B}[T]$ yields  {\it valuations} $\text{trop}: K[\![T]\!] \ra V\mb{B}[T]$; that is, the norms $\text{trop}$ are in fact multiplicative. It follows that our tropicalization scheme for differential algebraic geometry is a {\it differential enhancement} in the sense of J. Giansiracusa and S. Mereta of the classical tropicalization associated to the trop valuation (cf. \cite{GM}). %{\color{red} 
In section \ref{Section_Initial_forms}, we extend the valuative theory  from $K[\![T]\!]$ to %the ring of polynomials 
$K[\![T]\!][x_{i,J}]$; and in Definition~\ref{initial_form}, we reconstruct the initial forms of \cite{FGH20+} in terms of these extensions and the relevancy relation introduced in section~\ref{section_Idempotent_Semiring}. 
%{\color{red} } in terms of these extensions and the relevancy relation introduced in section~\ref{section_Idempotent_Semiring}. 

%\medskip
In section~\ref{Section_TFT}, we introduce tropical DA varieties and we present our fundamental theorem \ref{EFT} of tropical differential algebraic geometry, which characterizes these varieties in three distinct ways. While Theorem~\ref{EFT} was previously proved in \cite{FGH20}, our approach is novel and emerges naturally from the theory developed in the preceding sections. Thinking of tropical DA varieties as the supports of power series solutions associated to differential ideals of $K[\![T]\!]\{x_1,\ldots,x_n\}$, it is likewise natural to ask for an analogue of Theorem~\ref{EFT} for differential ideals of  $K(\!(T)\!)\{x_1,\ldots,x_n\}$. 

%\medskip
In section \ref{Section_Fromringtofield}, we extend the valuative theory of section \ref{Section_QVIF} from $K[\![T]\!]$ to $K(\!(T)\!)$. The key technical ingredient is Proposition~\ref{Prop_MC}, which establishes that the semiring $V\mb{B}[T]$ of vertex polynomials is integral, so canonically injects in its semifield of fractions $V\mb{B}(T)$. In Theorem~\ref{EEFT}, %{\color{red}we partially extend the fundamental theorem to $K(\!(T)\!)$, by producing} 
we partially extend the fundamental theorem to $K(\!(T)\!)$, by producing $K(\!(T)\!)$-analogues of two of the three characterizations of tropical DA varieties in Theorem \ref{EFT}; namely, as a set of weight vectors $w \in \mb{B}[\![T]\!]^n$ that define monomial-free initial differential ideals, and as a set of (coordinate-wise) supports of solutions associated to a given system of algebraic differential equations. %{\color{red} It is worth mentioning that the proof presented here is essentially different from all the previous approaches, as it uses the extended valuative theory and the properties of  the relation of relevancy.} 
It is worth noting that the proof presented here diverges from previous approaches, as it is based on generalized valuation theory and the relevancy relation.
Theorem~\ref{prop_local_trop_basis}, meanwhile, generalizes a result of \cite{HG19final} that describes {\it tropical bases} of differential ideals with respect to prescribed support vectors %$\omega_m \in \mb{B}[\![T]\!]^n$)} when $m=1$ %in $\mb{B}[\![T]\!]^n$ 
when $m=1$ to the case of arbitrary $m$. Subsection~\ref{initial_degenerations} is devoted to {\it initial degenerations} defined by weight vectors $w \in \mb{B}[\![T]\!]^n$, which are the geometric counterparts of ideals of initial forms. The  valuation $\text{trop}:K[\![T]\!] \ra V\mb{B}[T]$ introduced in section~\ref{Section_QVIF} extends naturally to a valuation $\text{trop}:K(\!(T)\!) \ra V\mb{B}(T)$, which in turn enables us to define its {\it ring of integers} $K(\!(T)\!)^{\circ} \sub K(\!(T)\!)$. For every weight vector $w \in \mb{B}[\![T]\!]^n$, we define a ``translation" by $w$ specified explicitly by \eqref{translation_map} that sends differential polynomials over $K(\!(T)\!)$ to differential polynomials over $K(\!(T)\!)^{\circ}$. Translating differential ideals by weight vectors, we obtain algebraic specialization maps that relate differential ideals to their initial forms, modulo a flatness hypothesis.

%\medskip
In section~\ref{Section_CA} we discuss computational aspects of tropical DA varieties, and we highlight several outstanding problems of valuation-theoretic and polyhedral nature. In the final subsection~\ref{Example} we evaluate a particular tropical differential polynomial $P$, and we describe a general strategy for computing the solution set of {\it any} $P$.

\subsection{Conventions} 
Outside of Section~\ref{tp}, every algebraic structure considered here will be commutative. We let $\mathbb{N}$ denote the semiring of natural numbers %(including $0$) 
endowed with the usual operations. %, and hereafter we also fix two a non-zero natural numbers $m$ and $n$. 
 Given a nonzero natural number $m$, we let $\{e_1,\ldots,e_m\}$ denote the usual basis of the free %commutative 
 $\mathbb{N}$-module $\mathbb{N}^m$. Given multi-indices $I=(i_1,\ldots,i_m)$ and $J=(j_1,\ldots,j_m) \in \mathbb{N}^m$, we set $\lVert I\rVert_\infty:=\text{max}_{}\{i_1,\ldots,i_m\}=\text{max}(I)$ and $J-I:=(j_1-i_1,\ldots,j_m-i_m)\in\mathbb{Z}^m$. Given any tuple $T=(t_1,\ldots,t_m)$ and any multi-index $I=(i_1,\ldots,i_m)\in \mathbb{N}^m$, we let $T^I$ denote the formal monomial $t_1^{i_1}\cdots t_m^{i_m}$. 
 Finally, whenever $A$ is a subset of $\mathbb{N}^m$, we set $A-I:=\{J-I \::\:J\in A\}$.
 
 \section{Differential algebraic geometry over power series semirings}
\label{Section_DifEqs}

In this section, we develop the algebraic architecture that underlies the semiring approach to differential algebraic geometry. Recall that a semiring is an algebraic structure which satisfies all the axioms of rings, possibly with the exception of the existence of additive inverses, and throughout this section, ${\rm S}=({\rm S},+,\times,0,1)$ will denote a  semiring with additive and multiplicative units 0 and 1, respectively. Semirings comprise a category, namely that of $\mathbb{N}$-algebras, that is a proper enlargement of the category of rings. The monograph \cite{JG} is a useful reference for the general theory of semirings.

%\medskip
 For tropical differential algebraic geometry, the central objects of study are differential polynomials with coefficients in a power series semiring, which generalize in a straightforward way the differential polynomials with coefficients in a differential field of characteristic zero studied in \cite{R50}. %just an adaptation from the case of rings from \cite{R50}. 
Algebraic differential equations in the traditional ring-theoretic sense are the zero loci of differential polynomials; unfortunately, as we shall see later, simply allowing all coefficients to belong to an arbitrary semiring ${\rm S}$ results in an unsatisfactory theory of tropical differential equations.

\begin{dfn}
{Fix an integer $m\geq1$. The semiring ${\rm S}[\![T]\!] =({\rm S}[\![T]\!] ,+,\times,0,1)$ of} formal power series {with coefficients in ${\rm S}$ in the variables $T=(t_1,\ldots,t_m)$ is the set of expressions $a=\sum_{I\in A}a_IT^I$ with $A\subseteq\mathbb{N}^m$ and $a_I\in S$, endowed with the usual operations of term-wise sum and convolution product.} 

\end{dfn}
\noindent That is, given $a=\sum_{I\in A}a_IT^I$ and $b=\sum_{I\in B}b_IT^I$, we have $a+ b=\sum_{I\in A\cup B}(a_I+b_I)T^I$ and $a b=\sum_{I\in A+ B}(\sum_{J+K=I}a_Jb_K)T^I$.

%\medskip
We let $\text{Supp}(a)$ denote the subset of the index set $A\subseteq\mathbb{N}^m$ in $a=\sum_{I\in A}a_I T^I$ whose associated coefficients $a_I$ are nonzero; this is the {\it support} of $a$. %When possible, we shall denote formal series with lowercase letters and their support sets by uppercase letters. 
Given $A\subseteq\mathbb{N}^m$, we let $e_{\rm S}(A)$ denote the series $\sum_{I\in A}T^I$ in ${\rm S}[\![T]\!] $.
%or just by $\sigma(A)$ if $S$ is clear from the context.

%\medskip
Let $n$ and $m$ be  nonzero natural numbers. We shall consider the polynomial ring ${\rm S}[\![T]\!][x_{i,J}\::\:i,J]$ where $i\in\{1,\ldots,n\}$ and $J\in \mathbb{N}^m$. The {\it order} of $x_{i,J}$ is $\lVert J\rVert_\infty$ and a  monomial  of order less than or equal to $r\in\mathbb{N}$ is any expression of the form
\begin{equation}
\label{differential_monomial}
    E:=\prod_{{1\leq i\leq n,\; ||J||_\infty\leq r}}x_{i,J}^{m_{i,J}}.  
\end{equation}

We may specify the  monomial \eqref{differential_monomial} using the array $M:=(m_{i,J}) \in \mathbb{N}^{n\times(r+1)^m}$, in which case we write $E=E_M$. With this notation, an element $P\in {\rm S}[\![T]\!][x_{i,J}\::\:i,J]$ is a finite sum $P=\sum_{i}a_{M_i}E_{M_i}$
 of scalar multiples of differential monomials   with nonzero coefficients $a_{M_i}\in {\rm S}[\![T]\!] $.

\subsection{Differential semirings}
Our next task is to equip these semirings with a differential structure.
\begin{dfn}
{A map $d:{\rm S}\rightarrow {\rm S}$ is a} \emph{derivation} {if it is linear and satisfies the Leibniz rule with respect to products of elements in ${\rm S}$. A} differential semiring {is a pair $({\rm S},D)$ consisting of a semiring ${\rm S}$ together with a finite collection $D$ of derivations on ${\rm S}$ that commute pairwise}.
\end{dfn}

The usual definition of partial derivations in a ring of formal power series with ring coefficients extends  to the semiring coefficient case. Namely, given $a\in {\rm S}\setminus\{0\}$, $I=(i_1,\ldots,i_m)\in \mathbb{N} ^m$ and $j=1,\ldots,m$, we set
\begin{equation}\label{rule_deriv}
\tfrac{\partial}{\partial t_j}(aT^I):=
\begin{cases}
i_jaT^{I-e_j}&\text{if }i_j\neq0; \text{ and}\\
0&\text{otherwise}.
\end{cases}
\end{equation} 

That \eqref{rule_deriv} is well-defined follows from the way in which $\mathbb{N}$ acts on ${\rm S}$. %{\color{red}this is well defined in the idempotent case since $n$ means $1+1+\cdots+1$ ($n$ times), see Section \ref{Tropical}.} 
Moreover, we have $\tfrac{\partial}{\partial t_j}(aT^IbT^J)=aT^I\tfrac{\partial}{\partial t_j}(bT^J)+bT^J\tfrac{\partial}{\partial t_j}(aT^I)$; by linearity it follows that $\tfrac{\partial}{\partial t_j}(\sum_{I}a_IT^I)=\sum_{I}\tfrac{\partial}{\partial t_j}(a_IT^I)$ defines a derivation on ${\rm S}[\![T]\!]$. 

We now extend $D=\{\tfrac{\partial}{\partial t_1},\ldots,\tfrac{\partial}{\partial t_m}\}$ from ${\rm S}[\![T]\!] $ to ${\rm S}[\![T]\!][x_{i,J}:i,J]$ by setting $\tfrac{\partial}{\partial_{t_i}} x_{k,J}= x_{k,J+e_i}$.  %We denote the resulting pair $({\rm S}[\![T]\!][x_{i,J}:i,J],D)$ by either ${\rm S}[\![T]\!] \{x_1,\ldots,x_n\}$ or  ${\rm S}_{m,n}$. 
Hereafter we use either ${\rm S}[\![T]\!] \{x_1,\ldots,x_n\}$ or  ${\rm S}_{m,n}$ as a shorthand for $({\rm S}[\![T]\!][x_{i,J}:i,J],D)$. We further set ${\rm S}_{0,0}:=({\rm S},D=\{0\})$, ${\rm S}_{m,0}:={\rm S}_m=({\rm S}[\![T]\!] ,D)$, and ${\rm S}_{0,n}:=({\rm S}\{x_1,\ldots,x_n\},D)$. The inclusion ${\rm S}_{a,b}\subset {\rm S}_{c,d}$ is an extension of differential semirings whenever $0\leq a\leq c$ and $0\leq b\leq d$. 

The following map is the key to interpreting elements of ${\rm S}_{m,n}$ as differential operators with coefficients in ${\rm S}$.  
%these differential polynomials as  differential equations with coefficients in $S$.

\begin{dfn}
  {Given a differential semiring $(S,D=\{d_1,\ldots,d_m\})$, let $\Theta_{({\rm S},D)}:\mathbb{N}^m\rightarrow {\rm End}({\rm S})$ denote the $\mathbb{N}$-module action in which $e_i$ acts as $d_i$. That is, the endomorphism of ${\rm S}$ corresponding to $(j_1,\ldots,j_m)=J  \in \mathbb{N}^m$ is $\Theta_{({\rm S},D)}(J):= d_1^{j_1}\cdots d_m^{j_m}$. Whenever there is no risk of confusion, we will write $\Theta(J)$ in place of $\Theta_{({\rm S},D)}(J)$.}
  %denote this map $\Theta(J)$ if there is no risk of confusion.
\end{dfn}

Each differential monomial $E$ as in \eqref{differential_monomial} singles out an evaluation map $E: {\rm S}_m ^n\rightarrow {\rm S}_m $ %, as the indices $J  \in \mathbb{N}^m$  of the variables $x_{i,J}$ of our differential polynomials encode the map $\Theta_{{\rm S}_m}(J):{\rm S}_m \to {\rm S}_m $. It 
that sends $a=(a_1,\ldots,a_n)\in {\rm S}_m ^n$ to

\begin{equation}\label{evaluation_map}
 E(a):=\prod_{m_{i,J}}(\Theta_{{\rm S}_m}(J)a_i)^{m_{i,J}}=\prod_{m_{i,J}}\bigl(\tfrac{\partial^{\sum_ij_i}(a_i)}{\partial t_1^{j_1}\cdots \partial t_m^{j_m}}\bigr)^{m_{i,J}}.
\end{equation}

The assignment \eqref{evaluation_map} extends by linearity to yield an evaluation map $P: {\rm S}_m ^n\rightarrow {\rm S}_m $ for every $P \in {\rm S}_{m,n}$; and deciding when $a\in {\rm S}_m ^n$ should be deemed a {\it solution} of $P$ will depend on the value $P(a)$. We will see later that the precise definition of solution we adopt will depend on the type of base semiring ${\rm S}$ under consideration.

\begin{rem}
\label{commutation}
{Let $\Theta_{{\rm S}_{m,n}}:\mathbb{N}^m\rightarrow {\rm S}_{m,n}$ denote the   $\mathbb{N}$-module action of $\mathbb{N}^m$ on ${\rm S}_{m,n}$. Because ${\rm S}_{m,0}\subset {\rm S}_{m,n}$ is an extension of differential semirings, it follows from \eqref{evaluation_map} that differentiation of differential polynomials commutes with evaluation, i.e., we have $(\Theta_{{\rm S}_{m,n}}(J) P)(a)=\Theta_{{\rm S}_m}(J) (P(a))$ for every $P\in {\rm S}_{m,n}$, $a\in {\rm S}_m ^n$ and $J  \in \mathbb{N}^m$.
}
\end{rem}

\subsection{Intermezzo: the semiring of twisted polynomials}
\label{tp}

In this subsection, we explain how to adapt differential polynomials to the semiring-theoretic context using differential modules over differential semirings. The resulting formalism is of a piece with contemporary presentations of differential algebra; see, e.g.,  \cite{K10}.

%\medskip
To begin, recall from \cite[Ch.14]{JG} that modules over a (not-necessarily commutative) semiring ${\rm S}$ are pairs $({\rm M},\cdot)$ consisting of a commutative monoid ${\rm M}=({\rm M},+,0_M)$ together with a scalar multiplication $\cdot:{\rm S}\times {\rm M}\rightarrow {\rm M}$ that satisfies the usual axioms for modules over rings, along with the additional requirement that $s\cdot 0_{\rm M}=0_{\rm M}=0_R\cdot m$ for every $s \in {\rm S}$ and $m \in {\rm M}$.

%\medskip
Similarly, the definition of differential modules over a differential semiring is an adaptation of the usual one for differential rings. We will focus on the case of ${\rm S}_m$. 

\begin{dfn}
{A }\emph{differential module} {over ${\rm S}_m$ is a pair $({\rm M},D_{\rm M})$ in which ${\rm M}$ is a module over ${\rm S}[\![T]\!] $ equipped with additive maps $D_{\rm M}=\{\delta_1,\ldots,\delta_m\}$ for which }$$\delta_i(a\cdot m)=a\cdot\delta_i(m)+\tfrac{\partial a}{\partial t_i}\cdot m \quad \text{ for every } a\in {\rm S}[\![T]\!] ,\:m\in {\rm M}.$$
\end{dfn}

\begin{dfn}
{The }\emph{semiring of twisted polynomials} {${\rm S}[\![T]\!] \{d_1,\ldots,d_m\}$ is the additive semigroup $\bigoplus_{I\in\mathbb{N}^m}{\rm S}[\![T]\!] \cdot D^I$, where $D^I=d_1^{i_1}\cdots d_m^{i_m}$, and in which we impose the following rules for the product: 
\begin{enumerate}
\item $d_ia=ad_i+\tfrac{\partial a}{\partial t_i}$ for every $a\in {\rm S}[\![T]\!] $; and
\item $d_id_j=d_jd_i$ for every $i,j=1,\ldots,m$.
\end{enumerate}
}
\end{dfn}

Via the semiring of twisted polynomials ${\rm S}[\![T]\!] \{d_1,\ldots,d_m\}$, we may identify the category of differential modules over ${\rm S}_m$ with the category of left ${\rm S}[\![T]\!] \{d_1,\ldots,d_m\}$-modules. Indeed, we have an action of ${\rm S}[\![T]\!] \{d_1,\ldots,d_m\}$ on $({\rm M},D_{\rm M})$ for which $d_i$ acts as $\delta_i$; so ${\rm S}[\![T]\!] \{d_1,\ldots,d_m\}$ acts on ${\rm S}_{m,n}$ and its differential sub-semiring ${\rm S}_m$ with $d_i$ acting as $\tfrac{\partial }{\partial t_i}$. In this way, we recover the actions $\Theta_{{\rm S}_{m,n}}$ and $\Theta_{{\rm S}_{m}}$ described before.

\section{Tropical differential equations and their solutions}
\label{Tropical}

%In this section, we study the case of tropical differential polynomials and we introduce the notion of formal power series solution for them. 
We now introduce tropical differential polynomials and the formal power series solutions to the tropical differential equations they define. In practice, this means applying the constructions of section~\ref{Section_DifEqs} to the case in which ${\rm S}$ is the boolean semiring $\mathbb{B}:=(\{0,1\},+,\times)$. Here $\times$ denotes the usual product, while $a+b=1$ whenever $a$ or $b$ is nonzero. 

%\medskip
A semiring ${\rm S}$ is {\it additively idempotent} if $a+a=a$ for all $a\in {\rm S}$, and  the category of idempotent semirings is precisely the category of $\mathbb{B}$-algebras. Thus the semiring structure map $\mathbb{N}\rightarrow {\rm S}$ factors through the idempotent semiring structure map $\mathbb{B}\rightarrow {\rm S}$ via the (unique) homomorphism $\mathbb{N}\rightarrow \mathbb{B}$. Hereafter, any reference to an idempotent semiring means one that is additively idempotent.

Most of the following concepts were introduced in \cite{FGH20} using subsets of $\mathbb{N}^m$; here we reframe them using the language of formal power series with coefficients in $\mathbb{B}$ in order to maximize their proximity to their classical counterparts; see Remark~\ref{rem_difference}. We refer the curious to the Example ~\ref{Ex_Solution}, or to the more involved worked example in section~\ref{Example}, to see how these concepts apply in practice.

The semiring $\mathbb{B}[\![T]\!]$ consists of the set
$\mathbb{B}[\![T]\!] =\bigl\{a=\sum_{I\in A}T^I\::\:\emptyset\subseteq A\subseteq\mathbb{N}^m\bigr\}$ endowed with the operations $a+ b=\sum_{I\in A\cup B}T^I$ and $a b=\sum_{I\in A+ B}T^I$ for  $a=\sum_{I\in A}T^I$ and $b=\sum_{I\in B}T^I$. Recall that $\mathbb{B}_m=(\mathbb{B}[\![T]\!] ,D)$, where $D=\{\tfrac{\partial}{\partial t_1},\ldots,\tfrac{\partial}{\partial t_m}\}$; in particular, for  %$(j_1,\ldots,j_m)=J  \in \mathbb{N}^m$ 
$J  \in \mathbb{N}^m$, the map
 $\Theta_{\mathbb{B}_m}(J): {\mathbb{B}}[\![T]\!] \to {\mathbb{B}}[\![T]\!] $ sends $a=\sum_{I\in A}T^I$ to the series $\Theta_{\mathbb{B}_m}(J)a=\sum_{I\in (A-J)_{\geq0}}T^I$, where 
\begin{equation}
\label{differential_sets}
    (A-J)_{\geq0}:=\left\{ (i_1,\ldots,i_m)\in A-J \:|\:i_1,\ldots,i_m\geq0\right\}.
\end{equation}
Indeed, this follows from the fact that the action of $\mathbb{N}$ used to define \eqref{rule_deriv}  factors through the structure map $\mathbb{B}\rightarrow \mathbb{B}[\![T]\!]$.

\begin{rem}
\label{rem_difference}
{Let $\mathcal{P}(\mathbb{N}^m)$ denote the %idempotent 
semiring whose elements are subsets of $\mathbb{N}^m$, in which the sum $A+B:=A\cup B$ is the union of sets and the product $AB:=\{I+J\::\:I\in A, J\in B\}$ is the Minkowski product. We use \eqref{differential_sets} to equip $\mathcal{P}(\mathbb{N}^m)$ with a $\mathbb{N}^m$-action that turns it into a differential semiring; in so doing, the support map $\text{Supp}:(\mathbb{B}[\![T]\!] ,D)\xrightarrow[]{}(\mathcal{P}(\mathbb{N}^m),D)$ becomes an isomorphism of differential semirings with inverse $A\mapsto e_{\mathbb{B}}(A)$. It is worth noting here that  $a=\sum_{I\in A}T^I\in \mathbb{B}[\![T]\!] $ and $A\in \mathcal{P}(\mathbb{N}^m)$ are distinct representations of the same object. However, we will see that the representation given by the semiring of supports $\mathcal{P}(\mathbb{N}^m)$  is more suitable for computations; see, e.g., the concrete description of the expression \eqref{differential_sets} in \eqref{concrete_description_sup_der}.}
%support $(A-J)_{\geq0}$ of $\Theta_{\mathbb{B}_m}(J)a$ }
\end{rem}

Our aim is now to use the elements of the differential semiring $\mathbb{B}_{m,n}$ 
%=(\mathbb{B}[\![T]\!] [x_{i,J}\::\:1\leq i\leq n, J\in \mathbb{N}^m],D)$ 
%of differential polynomials with coefficients in $\mathbb{B}[\![T]\!] $ 
to define tropical differential equations. This requires a bit of care, however, inasmuch as directly copying the definition operative in the ring-theoretic setting produces inadequate results.
Indeed, given $P=\sum_ia_{M_i}E_{M_i}$ in $\mathbb{B}_{m,n}$, the only elements $a\in \mathbb{B}_m ^n$ whose $P$-evaluations \eqref{evaluation_map} satisfy $P(a)=0$ are such that $E_{M_i}(a)=0$ for every $i$; see section~\ref{Example}.

So we will have to work harder in order to produce a useful definition of solution. %These will turn out to be intimately related to (Newton) polyhedra.
 %{\color{blue} Maybe recommend some source?}
For this purpose, we will make use of the following basic concepts from convex geometry. A {\it polyhedron} $P$ is the intersection of finitely many affine halfspaces in $\mathbb{R}^m$, and it is a {\it polytope} whenever it is bounded. A (proper) {\it face} of the polyhedron $P$ is the intersection of $P$ with a hyperplane $H$ that contains $P$ in one of its two half-spaces. We let $\mathcal{V}(P)$ denote the set of zero-dimensional faces, or {\it vertices} of $P$; and
%and we call them the extreme points (or vertices) of $P$. 
we let $\text{Conv}(A)$ denote the convex hull of a set $A\subset \mathbb{N}^m$. 
 %and by $A+B$ the Minkowski sum of two sets. 
 A polyhedron $P$ is {\it integral} whenever $P=\text{Conv}(P\cap\mathbb{Z}^m)$.
 
\begin{dfn}
\label{Def_New_poly}
{Given $A\subset \mathbb{N}^m$, its} Newton polyhedron %\footnote{This is also called Newton polygon in other sources.}
{$\text{New}(A)$ is the convex hull of the set $\{I+J:I\in A,J\in\mathbb{N}^m\}\subseteq\mathbb{R}_{\geq0}^m$. We say that $A^{\pr}\subseteq A$ is a} spanning set {whenever $\text{New}(A^{\pr})=\text{New}(A)$}.
\end{dfn}

\begin{thm}[See \cite{FGH20}]\label{vertex_poly} {Every $A\subset \mathbb{N}^m$ has $\text{Vert}(A):=\mathcal{V}(\text{New}(A))$ as its (unique) minimal spanning set. In particular, its minimal spanning set is finite}.
   \end{thm}

Given $a=\sum_{I\in A}T^I$ in $\mathbb{B}[\![T]\!] $ we let $V(a):=e_{\mathbb{B}}(\text{Vert}(A))$ denote the {\it vertex polynomial} of $a$. We now have the ingredients necessary to construct solutions of tropical differential polynomials. The following definition differs slightly from the original in \cite{FGH20}, but it is equivalent and has the advantage of being less technical.
 
\begin{dfn}
\label{Def_trop_vanishing}
{An element $a\in \mathbb{B}_m ^n$ is a} solution {of $\sum_{i}a_{M_i}E_{M_i}=P\in \mathbb{B}_{m,n}$ if for every monomial $T^{I}$ of $V(P(a))$, there are at least two distinct terms $a_{M_k}E_{M_k}$ and $a_{M_\ell}E_{M_\ell}$ of $P$ for which $T^{I}$ appears in both $V(a_{M_k}E_{M_k}(a))$ and $V(a_{M_\ell}E_{M_\ell}(a))$}.
\end{dfn}

It is worth emphasizing that to verify whether an $n$-tuple $a\in \mathbb{B}_m ^n$ of (a priori infinite) boolean power series %(a priori, infinite objects) 
satisfies the tropical differential equation determined by a differential polynomial, only a {\it finite} number of conditions need to be checked. 

%\medskip
 We let $\text{Sol}(P)\subset\mathbb{B}_m ^n$ denote the set of solutions of $P$; more generally, given any  collection $\Sigma\subset \mathbb{B}_{m,n}$ of tropical differential polynomials, we let $\text{Sol}(\Sigma)=\bigcap_{P\in \Sigma}\text{Sol}(P)$ denote the set of solutions common to all $P\in \Sigma$.
 
 A natural question that arises is how the tropical vanishing condition introduced in Definition \ref{Def_trop_vanishing} compares with the usual one operative in tropical geometry, in which $a\in \mathbb{B}_m ^n$ is a solution of $\sum_{i}a_{M_i}E_{M_i}=P\in \mathbb{B}_{m,n}$ provided there are at least two distinct terms $a_{M_k}E_{M_k}$ and $a_{M_\ell}E_{M_\ell}$ of $P$ for which $V(P(a))=V(a_{M_k}E_{M_k}(a))=V(a_{M_\ell}E_{M_\ell}(a))$. In the ordinary differential case, the vertex polynomial $V(\sum_{i\in A}t^i)$ is simply the monomial $t^{\text{min}(A)}$, and our definition of tropical vanishing agrees with the classical one. However when $m>1$, classical solutions are also solutions in our sense, while the converse fails to hold in general, as in the following example.

\begin{figure}[!htb]
    \centering
    \includegraphics[scale=0.75]{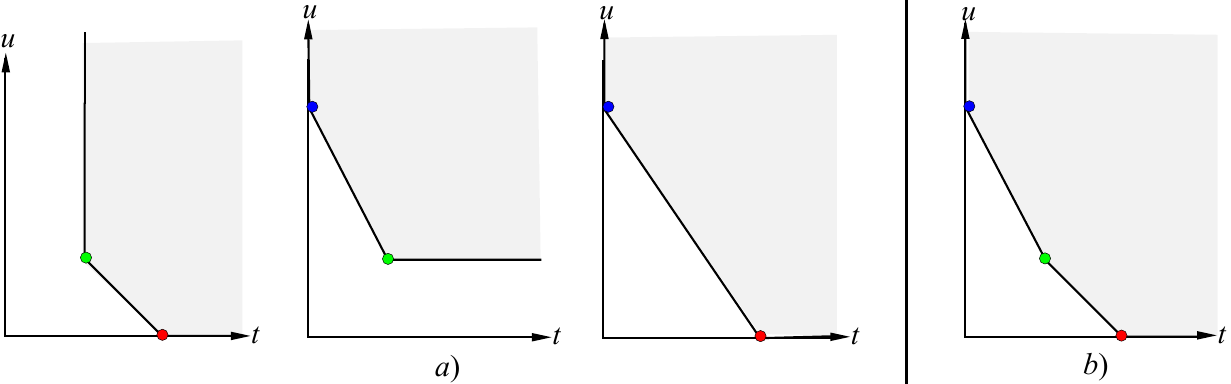}
    \vspace{-.3cm}
    \caption{%Newton polygons together with their vertex polynomials 
    a)  $V(a_{M_i}E_{M_i}(a))$ for $i=1,2,3$; and b)  $V(P(a))$.}
\label{Figura_soluciones}
\end{figure}

\begin{ex}
\label{Ex_Solution}
{
Set
$P=tx_{(1,0)}+ux_{(0,1)}+(t^2+u^3) \in \mathbb{B}[[t,u]]\{x\}$, and $a=t^2+tu+u^3\in \mathbb{B}[[t,u]]$. Then  $P(a)=t(t+u)+u(t+u^2)+(t^2+u^3)=t^2+tu+u^3$, and from Figure~\ref{Figura_soluciones}b we see that $V(P(a))=P(a)$.}

%\medskip
{
For every monomial $a_ME_M$ in the expansion of $P$, we now compute $a_ME_M(a)$ and $V(a_ME_M(a))$; here both series are equal in every instance, as is clear from Figure \ref{Figura_soluciones}a. We also see that $a$ is a solution of $P$ by directly checking against Definition \ref{Def_trop_vanishing}; however, there is no monomial that appears in both $V(a_ME_M(a))$ and $V(P(a))$. %are distinct for every monomial.
}
\end{ex}

Vertex polynomials are the key to this theory; we analyze them closely in the next section.

\section{The semiring of vertex polynomials}
\label{section_Idempotent_Semiring}
In section~\ref{Tropical}, the concept of vertex polynomial served primarily to define solutions of tropical differential equations; however, they fit into a broader algebraic framework. In this section, we will start by giving a more algebraic presentation of the semiring of vertex polynomials.

To this end, recall that a  {\it congruence} on a semiring ${\rm S}$ is an equivalence relation $\sim$ on ${\rm S}$ %\subset {\rm S}\times {\rm S}$ 
for which $a+c\sim b+c$ and $ac\sim bc$ for every $c\in {\rm S}$ whenever $a\sim b$. The semiring structure on ${\rm S}$ then descends to the quotient $\faktor{{\rm S}}{\sim}$, 
%set of equivalence classes $S/\sim$ becomes a semiring when equipped with the usual operations on equivalence classes: $[a]+[b]=[a+b]$ and $[a][b]=[ab]$, 
and the quotient projection $\pi:{\rm S}\to \faktor{{\rm S}}{\sim}$ becomes a homomorphism of semirings. 

\begin{dfn}
\label{Def_VS}
The \emph{semiring of vertex polynomials}
%$V\mathbb{B}[T]$ 
is the quotient $\faktor{\mathbb{B}[\![T]\!] }{\text{New}}$, where $\text{New}\subset \mathbb{B}[\![T]\!] \times \mathbb{B}[\![T]\!] $ denotes the semiring congruence comprised of pairs of boolean power series with equal Newton polyhedra. 
%We denote by $V:\mathbb{B}[\![T]\!] \rightarrow V\mathbb{B}[T]$ the quotient homomorphism.
\end{dfn}

%{\color{red} Again, most of these concepts were introduced in \cite{FGH20} using subsets of $\mathbb{N}^m$; for instance, }
As mentioned in the previous section, vertex polynomials were introduced in an equivalent form in \cite{FGH20} using subsets of $\mathbb{N}^m$; there,
the map $\text{Vert}:\mathcal{P}(\mathbb{N}^m)\xrightarrow[]{}\mathcal{P}(\mathbb{N}^m)$   was shown to satisfy $\text{Vert}^2=\text{Id}$; and the semiring of vertex sets $\mathbb{T}_m=(\mathbb{T}_m,\oplus,\odot)$ was defined to be the image of  $\mathrm{Vert}$, equipped with the operations $A\oplus B=\mathrm{Vert}(A\cup B)$ and $\ A\odot B=\mathrm{Vert}(AB)$. We will show below that $\mathrm{Vert}:\mathcal{P}(\mathbb{N}^m)\rightarrow \mathbb{T}_m$ is precisely the quotient projection $\pi:\mathbb{B}[\![T]\!] \rightarrow \faktor{\mathbb{B}[\![T]\!] }{\text{New}}$. This provides a more intrinsic explanation for the facts that $\mathbb{T}_m$ is a semiring, and that $\mathrm{Vert}$ is a semiring homomorphism.

\begin{lemma}\label{New_congruence}The  relation $\text{New}$  from Definition \ref{Def_VS} is a semiring congruence, and the map $\text{Vert}:\mathcal{P}(\mathbb{N}^m)\xrightarrow[]{}\mathbb{T}_m$  is the quotient projection $\pi:\mathbb{B}[\![T]\!] \rightarrow \faktor{\mathbb{B}[\![T]\!] }{\text{New}}$.
\end{lemma}
\begin{proof}
We use the isomorphism $\text{Supp}:\mathbb{B}[\![T]\!] \xrightarrow[]{}\mathcal{P}(\mathbb{N}^m)$ from Remark \ref{rem_difference}. Clearly $\text{New}$ is an equivalence relation. That $\text{New}$ is compatible with sum and product on $\mathbb{B}[\![T]\!] $ follows from the facts that $\text{New}(a+b)=\text{Conv}(\text{New}(a)\cup\text{New}(b))$ and $\text{New}(ab)=\text{New}(a)\text{New}(b)$.
%\noindent 
On the other hand, from \cite[Lem. 3.1]{FGH20} we have $\text{New}(a)=\text{New}(b)$ if and only if $\mathrm{Vert}(A)=\mathrm{Vert}(B)$, which means that $e_{\mathbb{B}}:\mathbb{T}_m\to \faktor{\mathbb{B}[\![T]\!] }{\text{New}}$ is a bijection of sets. Finally, the binary operations $\oplus, \odot$ on $\mathbb{T}_m$ satisfy $A\oplus B:=\mathrm{Vert}(A\cup B)=[e_{\mathbb{B}}(A)+e_{\mathbb{B}}(B)]$ and $A\odot B:=\mathrm{Vert}(A B)=[e_{\mathbb{B}}(A)e_{\mathbb{B}}(B)]$, which are precisely those of $\faktor{\mathbb{B}[\![T]\!] }{\text{New}}$.
\end{proof}

Hereafter $V\mathbb{B}[T]$ will denote the set of vertex polynomials, $i:V\mathbb{B}[T]\rightarrow \mathbb{B}[\![T]\!] $ the natural inclusion of sets induced by viewing polynomials as series, and $V:\mathbb{B}[\![T]\!] \rightarrow V\mathbb{B}[T]$ the quotient map. The operations on $V\mathbb{B}[T]$ are defined by $a\oplus b=V(i(a)+i(b))$ and $a\odot b=V(i(a)i(b))$. In particular, the support map $\text{Supp}:V\mathbb{B}[T]\xrightarrow[]{}\mathbb{T}_m$ is a semiring isomorphism.

Our next goal is to recast Definition~\ref{Def_trop_vanishing} in terms of corner loci of sums in $V\mathbb{B}[T]$. Given any sum $s=a_1\oplus\cdots\oplus a_k$ in $V\mathbb{B}[T]$ involving $k \geq 2$ summands, we let $s_{\widehat{i}}:=a_1\oplus\cdots\oplus\widehat{a_i}\oplus\cdots\oplus a_k$ denote the corresponding sum obtained by omitting the $i$-th summand, for every index $i=1, \dots,k$.

\begin{dfn}
{Given a positive integer $k \geq 2$, we say that the sum $s=a_1\oplus\cdots\oplus a_k$ in $V\mathbb{B}[T]$} \emph{tropically vanishes} {in $V\mathbb{B}[T]$ if $s= s_{\widehat{i}}\text{ for every }i=1,\ldots,k.$}
\end{dfn}

\begin{prop}
\label{prop_reformulation_solution}
An element $a\in\mathbb{B}_m ^n$ is a \emph{solution} of $P=\sum_{i=1}^ka_{M_i}E_{M_i} \in \mb{B}_{m,n}$
if and only if $V(P(a))=\bigoplus_{i=1}^kV(a_{M_i}E_{M_i}(a))$ tropically vanishes in $V\mathbb{B}[T]$.
\end{prop}

\begin{proof}
We have $P(a)=\sum_{i=1}^ka_{M_i}E_{M_i}(a)$, and as $V$ is a homomorphism, it follows that $V(P(a))=V(\sum_{i=1}^ka_{M_i}E_{M_i}(a))=\bigoplus_{i=1}^kV(a_{M_i}E_{M_i}(a)).$

Suppose now that $a\notin \text{Sol}(P)$; there is then some $t^I\in V(P(a))$ for which $t^I\in V(a_{M_i}E_{M_i}(a))$ for some unique index $i\in\{1,\ldots,k\}$, which means that $V(P(a))$ is not contained in $V(P(a))_{\widehat{i}}$. Conversely, if $V(P(a))$ does not tropically vanish, then $V(P(a))$ fails to be contained in $V(P(a))_{\widehat{i}}$ for some index $i\in\{1,\ldots,k\}$. This means there is some $t^I\in V(P(a))$ that appears as a term in the expansion of $V(a_{M_i}E_{M_i}(a))$ and in no other $V(a_{M_{\ell}}E_{M_{\ell}}(a))$, $\ell \neq i$;
%such that there aren't at least two different monomials $a_{M_k}E_{M_k}(a)$ and $a_{M_l}E_{M_l}(a)$ with $t^I$ in their vertex polynomials, 
thus $a\notin \text{Sol}(P)$.
\end{proof}

%{\color{red} Tropical vanishing of $V(P(a))$ is equivalent to  satisfy the bend relations from \cite[Sec. 5.1]{GG}, and the set of points satisfying them is precisely the corner locus of a tropical polynomial defined in the classical case.
%From Prop. \ref{prop_reformulation_solution} we see that $\text{Sol}(P)$ corresponds to the corner locus of the map $a\mapsto V(P(a))$, so it can be seen as a tropical hypersurface adapted to the differential case, we call it tropical DA hypersurface.

That $V(P(a))$ tropically vanishes means precisely that $V(P(a))$ satisfies the {\it bend relations} of \cite[Sec. 5.1]{GG};
so it is precisely the classical ``corner locus" of a tropical polynomial. Proposition~\ref{prop_reformulation_solution} establishes that $\text{Sol}(P)$ is the corner locus of the map $a\mapsto V(P(a))$, i.e., a ``tropical differential hypersurface"; correspondingly, we refer to $V(P(a))$ as a {\it tropical DA hypersurface}.
Note that tropical vanishing depends strongly on the %expression of the evaluation 
realization of $P(a)$ as a sum of monomials, but not on the {\it value} of $P(a)$; see Remark~\ref{rem_sum_not_value} for an illustrative example.

\subsection{The order relation}
\label{SubSection_Order}
In any idempotent semiring ${\rm S}$, we use $a\leq b$ to mean that $a+b=b$, or equivalently, that $a+c=b$ for some $c$ in ${\rm S}$. Then $\leq$ is an algebraic order relation, as $a\leq b$ implies $a+c\leq b+c$ and $ac\leq bc$ for all $c\in S$. It is a total order if and only if ${\rm S}$ is {\it bipotent}, i.e., such that $a+b \in \{a,b\}$ for every $a$ and $b$. In this context, the {\it least upper bound} of any finite set of elements of $S$ is equal to their sum. We use $a< b$ to mean that $a\leq b$ and $a\neq b$.

%{\color{red} If $S$ is an idempotent semiring, we denote by $S^{\text{op}}$ the same semiring endowed with the opposite order : $a\leq b$ if $a+b=a$. In this context,  the sum of any finite set of elements of $S$ becomes their {\it greatest lower bound}. We denote by $\ell:S\to S^{\text{op}}$ the identity map. This convention will be used later in order to describe the process of taking the non-archimedean amoeba. See Remark \ref{non_arch_amoeba}.}
Given an idempotent semiring ${\rm S}$, we let ${\rm S}^{\text{op}}$ denote the same semiring endowed with the opposite order: $a\leq b$ if and only if $a+b=a$. The sum of any finite set of elements in ${\rm S}$ then becomes their {\it greatest lower bound}. Hereafter, we let $\ell:{\rm S}\to {\rm S}^{\text{op}}$ denote the identity map. We will put this convention to use later in constructing non-archimedean amoebae; see Remark~\ref{non_arch_amoeba}.

We will now explicitly describe the  order relation on %the idempotent semiring 
$V\mathbb{B}[T]$. Interesting and as-yet unexplored combinatorial phenomena emerge whenever $m>1$. The following concept will be useful for our purposes.
%The following notions will be used later. We also try to make definitions easier and more combinatorial. New interesting combinatorial phenomena appears when we consider $m>1$. The following concept is useful in characterising the order.

  \begin{dfn}
  {
  Given $A\subset \mathbb{N}^m$, we let $\widetilde{A}:=\text{New}(A)\cap\mathbb{N}^m$ denote the set of integer points of the Newton polyhedron of $A$. We refer to $\mathbb{I}(A):=\widetilde{A}\setminus \text{Vert}(A)$ as the {\it irrelevant part} of $A$.}
  \end{dfn}
  \begin{lemma}\label{order_relation}
  Given $a,b\in V\mathbb{B}[T]$ with supports $A$ and $B$ respectively, we have $a\leq b$ if and only if $A\subseteq \widetilde{B}$.
  \end{lemma}
  \begin{proof}
  If $a\oplus b=b$, then $A\subset \text{Conv}(A\cup B+\mathbb{N}^m)=\text{New}(B)$, so $A=A\cap \mathbb{N}^m\subset \widetilde{B}$.
  
  Conversely, suppose that $A\subseteq \widetilde{B}$. As $B\subset \widetilde{B}$, we have $B\subset A\cup B\subset \widetilde{B}$, so $B+\mathbb{N}^m\subset A\cup B+\mathbb{N}^m\subset \widetilde{B}+\mathbb{N}^m=\widetilde{B}$. Since $\text{New}(B)$ is an integral polyhedron, it follows that $\text{Conv}(A\cup B+\mathbb{N}^m)=\text{New}(B)$, and consequently $\text{New}(B)=\text{Conv}(\widetilde{B})$. Passing to vertex sets now yields $a\oplus b=b$.
  \end{proof} 
  
  \begin{rem}
  %From the perspective of 
  {In the terminology of  %ordered sets
  \cite{Bl}, Lemma~\ref{order_relation} establishes that $V:\mathbb{B}[\![T]\!] \to V\mathbb{B}[T]$ is a {\it residuated mapping}, with residual
  $\wt{\cdot}:V\mathbb{B}[T]\xrightarrow{}\mb{B}[\![T]\!]$.}
  \end{rem}
  
 Note that $\leq$ defines a partial order on $V\mathbb{B}[T]$ with least (resp., largest) element $0$ (resp., 1). Here $V\mb{B}[T]$ contains the set of monomials $\{T^I: I\in\mb{N}^m\}$, and the order induced by $\leq$ on monomials is the {\it opposite} of the usual product order on $\mb{N}^m$. We will use Lemma~\ref{order_relation} to refine this partial order.
  \begin{dfn}\label{dfn_relevant}
 {
 Given $0<a\leq b$ in $V\mathbb{B}[T]$ with supports $A$ and $B$, $a$ is \emph{ irrelevant} for $b$ whenever $A\cap B=\emptyset$; in this case, we write $a\prec b$. }
 %Otherwise, $a$ is {\it relevant} for $b$, and we write $a\nprec b$.}
 %$A\cap B=\emptyset$; in this case, we write $a\prec b$. Otherwise, $a$ is {\it relevant} for $b$, and we write $a\nprec b$.}
 \end{dfn}
 
  Note that $a\prec b$ implies $a<b$, but the converse is not true unless $m=1$ or $m>1$ and $b=T^I$ is a monomial. The fact that either $a\nprec b$ or $a\prec b$ whenever $a \leq b$ generalizes the usual dichotomy between $a=b$ and $a<b$ that exists in totally ordered idempotent semirings whenever $a \leq b$.
  %{\color{red} The dichotomy $a\nprec b$ or $a\prec b$ whenever $a \leq b$ generalizes the usual dichotomy between $a=b$ and $a<b$ that exists in totally ordered idempotent semirings whenever $a \leq b$}. 
  In the following remark, we show that some properties of the  order relation $<$ on a totally ordered idempotent semiring ${\rm S}$ also hold for the relation $\prec$; see also Remark~\ref{rem_one_more}.
  
  \begin{rem}\label{properties_of_irrelevant}
{
Assume that $a,b, \text{and }c$ in $V\mathbb{B}[T]\setminus\{0\}$ have supports $A,B, \text{and }C$, respectively. The relation $\prec$ satisfies the following properties:
\begin{enumerate}
    \item {\it Transitivity:} $a\prec b, b\prec c \implies a\prec c$. This follows from the transitivity of $\leq$, and from $A\subset \mathbb{I}(B)\subset \widetilde{B}\subset \mathbb{I}(C)$.
    \item $c\prec a,b \implies c\prec a\oplus b$. Clearly $c\leq a\oplus b$,
    %and the conclusion follows since 
    and Proposition \ref{char_of_prod} implies that the support of $a\oplus b$ is contained in $A\cup B$, so  $C\cap(A\cup B)=\emptyset$. %because $C\cap A=C\cap B=\emptyset$.
     \item $a,b\prec c \implies a\oplus b\prec c$. Clearly $a\oplus b\leq c$, and the rest of the proof follows that of item (2).
    \item $a\prec b, c\neq 0 \implies a\odot c\prec b\odot c$. Clearly $a\odot c\leq b\odot c$, and the conclusion now follows from Proposition \ref{char_of_prod}. Indeed, otherwise would have $I_a+I_c=I=I_b+I_d$ for uniquely-prescribed $I_a\in A$, $I_b\in B$ and $I_c,I_d\in C$; but this is precluded by the fact that $A\cap B=\emptyset$.
    
    \item Cancellativity: $a\odot c\prec b\odot c$, with $c\neq 0 \implies a\prec b$. Here $a\odot c\leq b\odot c$ implies $a\leq b$, since $V\mathbb{B}[T]$ is cancellative by Proposition~\ref{Prop_MC}. If $I\in A\cap B$, then no sum of the form $I+J$ with $J\in C$ can appear in $b\odot c$ by hypothesis, but this contradicts the proof of Proposition \ref{Prop_MC}, which establishes that $I+J$ appears in $b\odot c$ for some $J\in C$. Thus $A\cap B=\emptyset$.
    %{\item \color{red}  if $a\odot c\prec b\odot c$ with $c\neq 0$, then $a\prec b$. The relation $a\odot c\leq b\odot c$ implies $a\leq b$ since $V\mathbb{B}[T]$ is cancellative by Proposition \ref{Prop_MC}. If $I\in A\cap B$, then no sum of the form $I+J$ with $J\in C$ can appear in $b\odot c$ by hypothesis, but this contradicts the proof of Proposition \ref{Prop_MC}, which says that $I+J$ appears in $b\odot c$ for some $J\in C$. Thus $A\cap B=\emptyset$.}
\end{enumerate}
}
\end{rem}
\subsection{Intermezzo: alternative characterizations of vertex polynomials}\label{alternative_characterizations}
In this subsection we focus on the polyhedral aspects of vertex polynomials. The Minkowski-Weyl theorem for polyhedra implies that for any $a\in \mathbb{B}[\![T]\!]$ there exists a polytope $\Delta$ for which
\begin{equation}\label{newton_decomposition}
    \text{New}(a)=\Delta+\mathbb{R}_{\geq0}^m.
\end{equation}
While %the second datum in the above representation is unique, the first one is not. 
$\Delta$ is not unique, there is a distinguished representative $\Delta=\Delta_a$ whose vertex set $\mathcal{V}(\Delta_a)$ is precisely $V(a)$; see \cite[Thm 4.1.3]{RW}.
We call $\Delta_a$ the {\it Newton polytope} of $a$. We anticipate that the Newton decomposition \eqref{newton_decomposition} may be leveraged to streamline the proofs of the salient properties of vertex sets described in \cite{FGH20}. Crucial among these is that vertex sets are {\it hereditary} in the following sense. 
%i.e., any subset of a vertex set is itself a vertex set.
\begin{prop}
\label{Prop_HProperty}
Given $a\in V\mathbb{B}[T]$ with support $A$, we have $e_{\mathbb{B}}(B)\in V\mathbb{B}[T]$ for every subset $B \sub A$.
\end{prop}

\begin{proof}
%Let $\text{New}(a)=\Delta_a+\mathbb{R}_{\geq0}^m$ and pick $B\subset A$. 
%We suppose that $B\neq\emptyset$ and $B\neq A$ otherwise there is nothing to do. 
Without loss of generality we assume that $\emptyset \subsetneq B \subsetneq A$, as otherwise the conclusion holds vacuously. We will show that $B=\mathcal{V}(\text{New}(B))$. To this end, first note that $B$ is the vertex set of $\Delta_B=\text{Conv}(B)$. Indeed, no $b \in B$ is realizable as a nontrivial convex combination $b=\sum_i\lambda_i b_i$ of distinct vertices $b_i \in B$, as no such nontrivial convex realization exists for vertices of $A$. Now $\text{New}(B)=\Delta_B+\mathbb{R}_{\geq0}^m$; so $\mathcal{V}(\text{New}(B))\subset B$.

%\medskip
On the other hand, the fact that $\text{New}(B)\subset \text{New}(A)$ implies that $\text{New}(B)\cap A\subset \mathcal{V}(\text{New}(B))$ by \cite[Lem. 3.1]{FGH20+}. Given $v\in \text{New}(B)\cap A$, write $v=w+x$ with $w\in\Delta_B\subset \Delta_A$ and $x\in \mathbb{R}_{\geq0}^m$. Since $v$ is a vertex of $\text{New}(a)$, there is a hyperplane $H$ such that $H\cap \text{New}(a)=\{v\}$. But since $w+\mathbb{R}_{\geq0}^m\subset \text{New}(a)$, we have $H\cap(w+\mathbb{R}_{\geq0}^m)\subset \{v\}$, which implies that $x=0$. As a result, we have $v=w$ and this forces $v\in B$.
\end{proof}

\begin{coro}
 Let $a,b\in V\mathbb{B}[T]$ with supports $A$ and $B$. If $0<a< b$, then there exist $a',b'\in V\mathbb{B}[T]$ with $i(a)=a'+b'$ and supports $A^{\pr}$ and $B^{\pr}$ satisfying $A^{\pr}\subsetneq B$ and $B^{\pr}\subsetneq \mathbb{I}(B)$.
\end{coro}
\begin{proof}
Applying Lemma~\ref{order_relation} and setting %in tandem with the decomposition 
$\wt{B}:=\text{Vert}(B)\sqcup \mb{I}(B)$, we have $A=A\cap \wt{B}=[A\cap \text{Vert}(B)]\sqcup [A\cap \mb{I}(B)]$, which is itself a union of vertex sets.
%by Proposition~\ref{Prop_HProperty}. 
\end{proof}

As further applications of the same circle of ideas, we will show that for $A\in \mathbb{N}^m$, the set $(\widetilde{A}\sqcup\{0\},+)$ is a finitely generated $\mathbb{N}$-module; and we will give a useful refined description of the supports of the algebraic operations in $V\mathbb{B}[T]$ in terms of vertices of lattice polytopes. In particular, \cite[Lem. 6]{FGH20+}, which establishes that for every $a,b\in V\mathbb{B}[T]$, every point $I$ in the support of $a\odot b$ is associated with uniquely-prescribed elements $I_a\in A$ and $I_b\in B$, is a consequence of this refined description.

\begin{prop}\label{finite_generation_A-tilde}
Let  $A\subset \mathbb{N}^m$ with $\Delta_A:=\text{Conv}(\text{Vert}(A))$. Then 
\begin{equation}\label{A-tilde_identity}
\widetilde{A}=(\Delta_A+[0,1)^m)\cap\mathbb{N}^m+\mathbb{N}^m.
\end{equation}

%We have $\widetilde{A}+\widetilde{A}\subset \widetilde{A}$ for every .  
%Furthermore,
%\begin{equation}\label{A-tilde_identity}
%\widetilde{A}=(\Delta_A+[0,1)^m)\cap\mathbb{N}^m+\mathbb{N}^m
%\end{equation}
%where $\Delta_A:=\text{Conv}(\text{Vert}(A))$ arises from the Minkowski decomposition $\text{New}(A)=\Delta_A+\mathbb{R}_{\geq0}^m$ of $\text{New}(A)$.
\end{prop}
\begin{proof}
Given $v,w\in\widetilde{A}$, we have $v=p+x$ and $w=q+y$ for some $p,q\in \Delta_A$ and $x,y\in \mathbb{R}_{\geq0}^m$. So $v+w=(p+q)/2+[(p+q)/2+x+y]\in\widetilde{A}$.

We now prove \eqref{A-tilde_identity}. The inclusion $\supseteq$ is clear, since any element $I$ on the right-hand side of \eqref{A-tilde_identity} may be written as $I=(I_1+I_2)+I_3$, where $(I_1+I_2)\in \Delta_A+[0,1]^m$ and $I_3\in \mathbb{N}^m$; so in particular $I=I_1+(I_2+I_3)$ with $I_1\in \Delta_A$ and $(I_2+I_3)\in \mathbb{R}_{\geq0}^m$. 

Conversely, given $v\in \widetilde{A}$, we have $v=p+x$ as before, and moreover $x$ decomposes as $x=m+y$ with $m\in\mathbb{N}^m$ and $y\in[0,1)^m$. So $v=p+m+y$, and $v-m=p+y\in (\Delta_A+[0,1]^m)\cap\mathbb{N}^m$ as $v$ and $m$ are integral.
\end{proof}

\subsection{Comparison with the semiring of lattice polytopes}
\label{Sub_Section_Relationship}
Let $\mathcal{P}_{\mathbb{Z}^m}=(\mathcal{P}_{\mathbb{Z}^m},\oplus,\otimes)$ denote the semiring of lattice polytopes in $\mathbb{Z}^m$, in which $\Delta_1 \oplus \Delta_2:=\text{Conv}(\Delta_1 \cup \Delta_2)$, and $\Delta_1 \otimes \Delta_2$ denotes the Minkowski sum of $\Delta_1$ and $\Delta_2$. This is an idempotent semiring in which $\Delta_1 \leq \Delta_2 $ if and only if $\Delta_1 \subseteq \Delta_2$.

\begin{prop}
\label{char_of_prod}
The map $\text{Conv}:V\mathbb{B}[T]\rightarrow\mathcal{P}_{\mathbb{Z}^m}$ that sends $a$ to $\text{Conv}(a)=\Delta_a$ is a non-archimedean norm in the sense of Definition~\ref{dfn:quasivaluation}.
\end{prop}
%The concept  of non-archimedean norm can be consulted in Definition \ref{dfn:quasivaluation}.
%By non-archimedean norm we mean that $\text{Conv}(a)=0$ if and only if $a=0$, $\text{Conv}(1)=1$, $\text{Conv}(a\odot b)\leq \text{Conv}(a)\otimes \text{Conv}(b)$ and $\text{Conv}(a\oplus b)\leq \text{Conv}(a)\oplus \text{Conv}(b)$ hold for every $a,b\in V\mathbb{B}[T]$.

\begin{proof}
The fact that a polytope is determined by its vertices implies that $\text{Conv}$ is injective. Now suppose that $a,b$ have supports $A$, $B$ and Newton polytopes $\Delta_a,\Delta_b$, respectively. By definition, we have $a\odot b=V(ab) =\mathcal{V}(\text{New}(A+B+\mathbb{N}^m))$. As
  $\text{New}(A+B+\mathbb{N}^m)=(\Delta_a\otimes\Delta_b)+\mathbb{R}_{\geq0}^m$, it follows that $\text{Conv}(a\odot b)\subset \Delta_a\otimes\Delta_b$.
  Similarly, we have $a\oplus b=V(a+b)=\mathcal{V}(\text{New}(A\cup B+\mathbb{N}^m))$; and as 
  $\text{New}(A\cup B+\mathbb{N}^m)=(\Delta_a \oplus \Delta_b)+\mathbb{R}_{\geq0}^m$, it follows that $\text{Conv}(a\oplus b)\subset \Delta_a \oplus \Delta_b$. In particular, we have chains of inclusions %the following chains 
  \begin{equation}
\label{uniqueness}
\begin{aligned}
V(ab)\subseteq \mathcal{V}(\Delta_a\otimes \Delta_b)\subseteq A+B \text{ and }V(a+b)\subseteq \mathcal{V}(\Delta_a\oplus \Delta_b)\subseteq A\cup B.
\end{aligned}
 \end{equation}
%and we have a description for $\mathcal{V}(\Delta_a\otimes\Delta_b)$ (see e.g., in \cite[\S 3.1]{DT}), namely 
On the other hand, according to \cite[\S 3.1]{DT}, we have 
\[
\mathcal{V}(\Delta_a\otimes\Delta_b)=\{I\in A+B: I \text{ decomposes {\it uniquely} as }I_a+I_b, I_a\in \Delta_a, I_b\in \Delta_b\}.
\]
 \end{proof}
 
 \begin{figure}[!htb]
    \centering
    \includegraphics[scale=0.75]{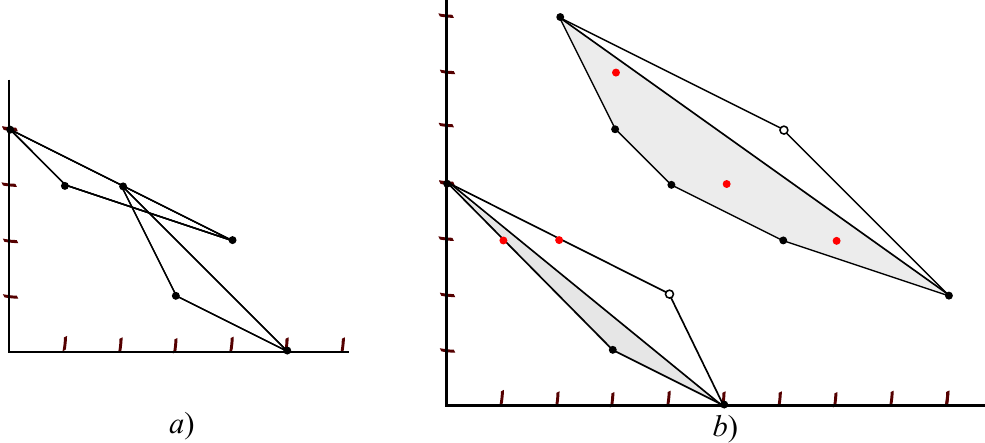}
    \caption{Polytopes from Example~\ref{polytope}: a) $\Delta_a$ and  $\Delta_b$; and b) $\text{Conv}(a\oplus b)\subsetneq \Delta_a\oplus\Delta_b$ and $\text{Conv}(a\odot b)\subsetneq \Delta_a\otimes\Delta_b$, where the included polytopes are shaded.}
\label{Figura_politopos}
\end{figure}
  
\begin{ex}\label{polytope}
{
%when $m>1$.
%Indeed, by considering the vertex sets 
Let $A=\{(2,3),(3,1),(5,0)\}$ and $B=\{(0,4),(1,3),(4,2)\}$; then $A\odot B=\{(2,7),(3,5),(4,4),(6,3),(9,2)\}$ and $A\oplus B=\{(0,4),(1,3),(5,0)\}$. So the inclusions \eqref{uniqueness} are proper in this case.}

%\medskip
\noindent {Now let 
$\Delta_a=\text{Conv}(A)$ and $\Delta_b=\text{Conv}(B)$.  As indicated in Figure~\ref{Figura_politopos}b, we have
\[
\begin{aligned}
\mathcal{V}(\Delta_a\otimes\Delta_b)&=A\odot B\cup\{(6,5)\}\text{ and }  A+B=\mathcal{V}(\Delta_a\otimes\Delta_b)\cup\{(3,6),(5,4),(7,3)\};  \\ 
\mathcal{V}(\Delta_a\oplus\Delta_b)&=A\oplus B\cup\{(4,2)\}\text{ and } A\cup B=\mathcal{V}(\Delta_a\oplus\Delta_b)\cup\{(1,3),(2,3)\}.
\end{aligned}
\]
}
\end{ex}

Now let $\overline{\mathbb{Z}}$ denote the idempotent semifield given set-theoretically by $\mathbb{Z}\cup\{-\infty\}$ and equipped with the usual max-plus tropical addition and multiplication laws, %(i.e., max and $+$, respectively), 
and let  $\mathcal{O}_{\mathbb{Z}^m}$ denote the idempotent semifield given set-theoretically by
$\mathcal{O}_{\mathbb{Z}^m}= \{f:\mathbb{Z}^m\rightarrow\overline{\mathbb{Z}} \::\:f\text{ is piecewise linear} \}\cup\{-\infty\}$ and equipped with the operations of {\it point-wise} tropical addition and multiplication. 

The map $\phi:\mathcal{P}_{\mathbb{Z}^m}\ra \mathcal{O}_{\mathbb{Z}^m}$ sending the polytope $P$ to its support function $\phi_P:\mathbb{Z}^{m}\rightarrow\overline{\mathbb{Z}}$ defined by $\phi_P(x)=\text{max}\{\langle u,x\rangle\::\:u\in P\}$   (where $\langle u,x\rangle$ denotes the usual Euclidean inner product) is an embedding of semirings. See \cite{KM19} for a recapitulation of the well-known correspondence between piecewise linear functions and polytopes. 
 
%\medskip
Given $a\in V\mathbb{B}[T]$, let $\phi_{a}:\mathbb{Z}^m\xrightarrow[]{}\mathbb{Z}$ denote the support function associated to $\Delta_a$.
%given explicitly by $\phi_{a}(x)=\text{max}\{\langle u,x\rangle\::\:u\in \Delta_a\}.$ 
If   $X\subset \mathcal{O}_{\mathbb{Z}^m}$ denotes the image of $V\mathbb{B}[T]$ under $\phi_a$, it is natural to wonder whether $X$ and $\overline{\mathbb{Z}}$ may be endowed with an algebraic structure so that the resulting function $V\mathbb{B}[T]\rightarrow X$ sending $a$ to $\phi_{a}$ is a semiring homomorphism.

\section{Non-Archimedean seminorms and initial forms}
 \label{Section_QVIF}
 
 Our tropicalization scheme for differential algebraic geometry is an enriched version of the classical tropicalization scheme based on valuations; the key technical notion we require is that of  
 non-Archimedean seminorm, which is closely 
 related to the quasivaluations of \cite{KM19} and the ring valuations of \cite{GG}.
 %concepts of quasivaluation from \cite{KM19}, and ring valuations from \cite{GG}.
 
 We also use non-Archimedean seminorms in order to single out new types of {\it initial forms} associated with (differential) polynomials.
Initial forms for differential polynomials $P\in K_{m,n}$ over a field of characteristic zero $K$ were first introduced in \cite[Sect. 2]{HG19} when $m=1$; they were subsequently generalized in \cite[Sect. 8]{FGH20+} for every $m$.  Here we first reframe initial forms in terms of seminorms, and in Section~\ref{Section_Fromringtofield} we show how to extend the construction of initial forms to differential polynomials with coefficients in  $\text{Frac}(K[\![T]\!])$. Throughout this section, ${K}$ denotes a field of characteristic zero.

\begin{dfn}\label{dfn:quasivaluation}
 {A }\emph{non-Archimedean seminorm} {is a map $|\cdot|:R\rightarrow {\rm S}$ from a semiring $R$ to an idempotent semiring ${\rm S}$ for which
 \begin{enumerate}
 \item $|0|=0$ and  $|1|=1$;
 \item $|a+b|\leq |a|+|b|$; and
 \item $|ab|\leq |a||b|$.
 \end{enumerate} 
Whenever $R$ is a ${K}$-algebra, we replace the condition $|1|=1$ by the requirement that $|a|=1$ for every nonzero element $a\in K$. The non-Archimedean seminorm $|\cdot|$ is } a norm {if furthermore no nonzero element $a$ satisfies $|a|=0$, and  }multiplicative {if it satisfies $|ab|= |a||b|$ for every $a, b \in R$. A multiplicative norm is called a valuation. Whenever ${\rm S}$ arises from a totally ordered monoid $({\rm M},\times,1,\leq)$, the non-Archimedean seminorm $|\cdot|$ is of {\it Krull} (or classical) type.}
 \end{dfn}

Hereafter we will simply speak of seminorms, with the implicit understanding that they are non-Archimedean.
%{\color{red}Since all of our seminorms will be non-Archimedean, we will just say seminorm.}

\begin{dfn}
 \label{defi_support}
{The }\emph{support series} {of the  element $a=\sum_{I\in A}a_IT^I\in {\rm S}[\![T]\!] $ is 
 the boolean formal power series $\text{sp}(a):=\sum_{I\in A}T^I$.}
 \end{dfn}
 
 When ${\rm S}=K$, the resulting map $\text{sp}:{K}[\![T]\!] \rightarrow\mathbb{B}[\![T]\!]$ is a norm, and the map $\text{Supp}:{K}[\![T]\!] \rightarrow\mathcal{P}(\mathbb{N}^m)$ from \cite{FGH20} fits into the following commutative diagram:
 \begin{equation*}
     \xymatrix{K[\![T]\!] \ar[r]^{\text{sp}}\ar[dr]_{\text{Supp}}&\mathbb{B}[\![T]\!]\ar[d]^{\ell} \\
     &\mathcal{P}(\mathbb{N}^m).}
 \end{equation*}
 
 Moreover, as we show in the proof of Proposition \ref{Longue_prop} below, the  norm $\text{sp}$ commutes with the respective differentials of $K_{m}=({K}[\![T]\!] ,D)$ and $\mathbb{B}_{m}=(\mathbb{B}[\![T]\!] ,D)$% , this is $\text{sp}\circ\Theta_{{K}_{m}}(e_i)= \Theta_{\mathbb{B}_{m}}(e_i)\circ \text{sp}$, for $i=1,\ldots,m$
. Putting everything together and applying Lemma~\ref{New_congruence}, we obtain a more intrinsic formulation of the following result from \cite[Sect. 4]{FGH20}.

\begin{thm}\label{VBT_theorem}
We have a commutative diagram
 \begin{equation}
 \label{Partial_diff_scheme}
 \xymatrix{
 &\mathbb{B}_m\ar[r]^-{\text{for}}&\mathbb{B}[\![T]\!]\ar[d]^{\pi}\\
 K_m\ar[r]_{\text{for}}\ar[ur]^{\text{sp}}&{K}[\![T]\!]\ar[ur]^{\text{sp}} \ar[r]_{\text{trop}}& \faktor{\mathbb{B}[\![T]\!] }{\text{New}}\\
 }
 \end{equation}
in which \emph{for} denotes the map that forgets the differential structures, while
\begin{enumerate}
\item sp is a ${K}$-algebra norm that commutes with $D$;
\item $\pi:\mathbb{B}[\![T]\!] \rightarrow \faktor{\mathbb{B}[\![T]\!] }{\text{New}}$ is a quotient homomorphism; and
\item $\text{trop}:K[\![T]\!] \rightarrow \faktor{\mathbb{B}[\![T]\!] }{\text{New}}$ is a ${K}$-algebra valuation.
 \end{enumerate}
 \end{thm}
 
  \begin{proof}
 Item 1 (resp., 2) follows  from Proposition~\ref{Longue_prop} (resp., Lemma~\ref{New_congruence}). Item 3 is the content of \cite[Lem. 4.3]{FGH20}, but the argument given there for the multiplicativity of trop unfortunately is lacking, as it is predicated on reducing to a multiplicative property for (pairs of) Newton polytopes in several variables, which fails in general.
 
 To justify item 3, we argue as follows. Given $I\in \trop(\varphi\psi)$, let $c_It^I$ denote the corresponding monomial. According to Proposition~\ref{char_of_prod}, there are unique monomials $a_JT^{J}$ of $\overline{\varphi}$ and $b_KT^{K}$ of $\overline{\psi}$ for which $c_It^I=a_JT^{J}b_KT^{K}$; so %even if the map sp is not multiplicative, the points $I\in \trop(\varphi\psi)$ are contained in 
$I \in \text{sp}(\overline{\varphi})\text{sp}(\overline{\psi})$; that trop is multiplicative now follows from the fact that $\pi$ is a homomorphism.
\end{proof}

The algebro-geometric study of tropical geometry over the semifield of piecewise linear functions $\mathcal{O}_{\mathbb{Z}^m}$, and in particular of $\mathcal{O}_{\mathbb{Z}^m}$-valued valuations, was initiated in \cite{KM19}. Remark~\ref{char_of_prod} %we showed that $V\mathbb{B}[T]$ is isomorphic to a subsemiring of $\mathcal{O}_{\mathbb{N}^m}$, thus 
implies that our valuations $\text{trop}:K[\![T]\!] \xrightarrow{}V\mathbb{B}[T]$ are closely related to theirs. It also establishes that the order structure of the idempotent semiring $V\mathbb{B}[T]$ is richer than that of $\mathcal{P}_{\mathbb{Z}^m}$. This will be relevant in section \ref{Section_Fromringtofield}.
 
 \begin{rem}
 \label{rem_diff_enh}
The diagram \eqref{Partial_diff_scheme} from Theorem~\ref{VBT_theorem} establishes that sp is an instance of what J. Giansiracusa and S. Mereta call a \emph{differential enhancement} of the trop valuation. See \cite[Section 4.7]{GM}
\end{rem}

It is important to note that although  $\text{trop}:K[\![T]\!]\xrightarrow{}V\mathbb{B}[T]$ is a  valuation, it is never of Krull type when $m>1$, which follows from  Lemma \ref{order_relation}. Remarkably, though, there is still a meaningful theory of initial forms with respect to trop, as we  show in the next subsection. 

\begin{rem}
\label{rem_one_more}
{Using the  relation of irrelevancy from Definition \ref{dfn_relevant} we show that if $\trop(a)\prec \trop(b)$, then $\trop(a+b)=\trop(a)\oplus\trop(b)=\trop(b)$, because $b$ decomposes as $b=\overline{b}+(b-\overline{b})$ (cf. \eqref{New_decomposition}), and in this instance no cancellations may occur because $\trop(a)\cap \trop(b)=\emptyset$.}
\end{rem}

\subsection{Initial forms}
\label{Section_Initial_forms}
%{\color{red} The purpose of this subsection is to present the following extension from $n=0$ to $n>0$ of Theorem \ref{VBT_theorem} in order to recast the results of \cite{FGH20+} in terms of valuation theory.
The main result of this subsection is the following extension from $n=0$ to $n>0$ of Theorem~\ref{VBT_theorem}; it is the key to reformulating the results of \cite{FGH20+} in terms of valuation theory.

\begin{thm}\label{VBT_theorem_extended}
Given $w=(w_1,\ldots,w_n)\in \mathbb{B}_m ^n$, we have a commutative diagram
 \begin{equation}
 \label{Ext_Partial_diff_scheme}
 \xymatrix{
 &\mathbb{B}_{m,n}\ar[r]^-{\text{for}}&\mathbb{B}[\![T]\!][x_{i,J}]\ar[r]^-{\text{ev}_w}&\mathbb{B}[\![T]\!]\ar[d]^{V}\\
 K_{m,n}\ar[r]_{\text{for}}\ar[ur]^{\text{sp}}&{K}[\![T]\!][x_{i,J}]\ar[ur]^{\text{sp}} \ar[rr]_{\text{trop}_w}&& V\mathbb{B}[T]\\
 }
 \end{equation}
in which \emph{for} denotes the map that forgets the differential structures, while
\begin{enumerate}
\item the map $\text{sp}:{K}_{m,n}\rightarrow \mathbb{B}_{m,n}$  is a  ${K}$-algebra norm  satisfying $\text{sp}\circ\Theta_{{K}_{m,n}}(e_i)\leq \Theta_{\mathbb{B}_{m,n}}(e_i)\circ \text{sp}$ for  $i=1,\dots,m$; and
\item $\text{trop}_w:{K}[\![T]\!][x_{i,J}] \rightarrow V\mathbb{B}[T]$ is a  multiplicative $K$-algebra seminorm.

 \end{enumerate}
 \end{thm}
 
%In this  we use the properties of the extension from $n=0$ to $n>0$ of the norm  $\text{sp}:{K}_{m,n}\rightarrow \mathbb{B}_{m,n}$ from Definition \ref{defi_support}     
The crucial technical tool operative in the proof of Theorem~\ref{VBT_theorem_extended} is the following result, whose proof we defer to Appendix~\ref{Proof_of_Longue_prop}.
%Our main technical tool is the following result, whose proof we defer to Appendix~\ref{Proof_of_Longue_prop}.

 \begin{prop}
 \label{Longue_prop}
For $n>0$, the map $\text{sp}:{K}_{m,n}\rightarrow \mathbb{B}_{m,n}$ that sends $P=\sum_{i=1}^d\alpha_{M_i}E_{M_i}$ to $\text{sp}(P)=\sum_{i=1}^d\text{sp}(\alpha_{M_i})E_{M_i}$ is a  ${K}$-algebra norm that satisfies $\text{sp}\circ\Theta_{{K}_{m,n}}(e_i)\leq \Theta_{\mathbb{B}_{m,n}}(e_i)\circ \text{sp}$ for  $i=1,\dots,m$.
\end{prop}

We now construct the seminorm $\trop_w$ for $w\in \mathbb{B}_m ^n$ in order to define the initial form of a differential polynomial $P\in K_{m,n}$ ``at" the support vector $w\in \mathbb{B}_m ^n$.  First we use the valuation $\trop:K[\![T]\!] \to{}V\mathbb{B}[T]$ to define the concept {\it initial term}. Note that any nonzero element $a=\sum_{I\in A}a_IT^I$ in $K[\![T]\!]$ and with $\text{Supp}(a)=A$ decomposes as 
\begin{equation}\label{New_decomposition}
    a=\overline{a}+(a-\overline{a}), \text{ where } \overline{a}=\sum_{\{I\::\:\trop(a_I)\nprec \trop(a)\}}a_IT^I=\sum_{I\in\text{Vert}(A)}a_IT^I
\end{equation}
 and $\overline{a}$ is the initial term of $a$ with respect to the valuation trop.
 
 We then use the section $e_K:\mathbb{B}[\![T]\!]\to K[\![T]\!]$ to obtain a $K$-lift  $e_K(w) \in K_m^n$
%=(\sigma_K(w_1),\ldots,\sigma_K(w_n))$ in $K_m^n$ 
of the vector $w=(w_1,\ldots,w_n)\in \mathbb{B}_m ^n$. For every  monomial $aE$ in ${K}[\![T]\!][x_{i,J}]$, we interpret it as a differential monomial and set
\begin{equation}\label{trop_of_diff_monomial}
\trop_w(aE):=\trop(aE(e_K(w)))=\trop(a)\odot  V(E(w)).
\end{equation}

\begin{dfn}\label{tropical_quasivaluation}
{Given $w\in \mathbb{B}_m ^n$, we define $\trop_w:{K}[\![T]\!][x_{i,J}]\rightarrow V\mathbb{B}[T]$  by extending the assignment \eqref{trop_of_diff_monomial} by linearity; explicitly, we have
%the unique $K[\![T]\!] $-module homomorphism given on differential monomials $aE\in K_{m,n}$ by $\trop_w(aE)=\trop(a)\odot V(E(w))$
%. Now we extend by linearity : if $P=\sum_Ma_ME_M$, then  
\begin{equation*}
\label{New_valuation}
    \trop_w\biggl(\sum_Ma_ME_M\biggr):=\bigoplus_M\trop_w(a_ME_M)=\bigoplus_M\bigl(\trop(a_M)\odot V(E_M(w))\bigr).
\end{equation*}
}
\end{dfn}

Note that  the definition of $\trop_w$ emulates the usual valuation specified by classical tropical algebraic geometry, and the outcome is  a multiplicative $K$-algebra  seminorm. We spell this out explicitly in Corollary~\ref{coro_Longue_prop} of Proposition \ref{Longue_prop}, whose proof we also defer to Appendix~\ref{Proof_of_Longue_prop}.

We now use the valuation $\trop_w$ together with the 
decomposition \eqref{New_decomposition} to define the {\it initial form}  $\text{in}^{\ast}_w(P)$ of $P\in {K}[\![T]\!][x_{i,J}]$ with respect to the support vector $w\in \mathbb{B}_m ^n$.
  
  \begin{dfn}\label{initial_form}
{
The }initial form {$\text{in}_w^{\ast}(P)$ of $P=\sum_Ma_ME_M$ at $w\in \mathbb{B}_m ^n$ is the differential polynomial}
\begin{equation}\label{initial_modified}
    \text{in}_w^{\ast}(P)=\sum_{\{M:\:\trop_w(a_ME_M)\nprec \trop_w(P)\}}\overline{a_M}E_M.
\end{equation}
\end{dfn}

The upshot of Remark~\ref{properties_of_irrelevant} is that every $P\in {K}[\![T]\!][x_{i,J}]$ decomposes as 
\[
    P=\text{in}_w^{\ast}(P)+P^{\ast}, \text{ in which }\trop_w(P^{\ast})\prec\trop_w(P)=\trop_w(\text{in}_w^{\ast}(P))
\]
where $P^{\ast}=P-\text{in}_w^{\ast}(P)$ is called {\it $w-$irrelevant} part of $P$.

\begin{ex}\label{initial_not_mult}
{The map $\text{in}_w^{\ast}: {K}[\![T]\!][x_{i,J}]\xrightarrow[]{}{K}[\![T]\!][x_{i,J}]$ fails to be multiplicative when $m>1$, as it restricts to the map ${K}[\![T]\!]\xrightarrow[]{}{K}[\![T]\!]$ that sends $a$ to $\overline{a}$, and the support of the product of two vertex polynomials is in general properly contained in the Minkowski sum of their supports (see Proposition~\ref{char_of_prod}).
}
%in light of Proposition \ref{char_of_prod}: given $a,b\in V\mathbb{B}[T]$ with supports $A$ and $B$, the support of $a\odot b$ is in general properly contained in $A+B$.}
  \end{ex}

   \begin{rem}
   
    Definition~\ref{initial_form} is a modified version of initial form from the classical setting; 
    %specified by \eqref{classic_initial}; %it is worth mentioning that 
    it also generalizes the definition of initial form given in the preprint \cite{HG19} when $m=1$, in which case trop is a Krull valuation. For a related construction, see \cite{FT20}.
    \end{rem}
    
     \section{A tropical fundamental theorem for differential algebraic geometry}
 \label{Section_TFT}
In this section we present the extended fundamental theorem in the differential context. Here $R={K}_{m,n}$, where ${K}$ is a field of characteristic zero; ${\rm S}=\mathbb{B}_{m,n}$ as in section~\ref{Tropical}; and the seminorm (valuation) is a support map. The differential enhancement of trop given by the diagram \eqref{Partial_diff_scheme} is the key to proving the fundamental theorem, which characterizes what we call {\it tropical DA varieties} in three distinct ways.
%\color{red} Due to the extra structure in the differential algebraic setting, we need to use the differential enrichment of trop given by the diagram \eqref{Partial_diff_scheme}.}
%We then formulate an adapted tropical fundamental theorem, and 
We'll discuss the usefulness of the tropical fundamental theorem in extracting combinatorial information about solutions to systems of differential equations.

 Given a system of differential polynomials $\Sigma$ in $K_{m,n}$, the best possible outcome would be to 
 %\subseteq {K}_{m,n}=({K}[\![T]\!] \{x_1,\ldots,x_n\},D)$, %we'd like to compute the set of $n$-tuples $\varphi=(\varphi_1,\ldots,\varphi_n)\in \mathbb{K}[\![T]\!] ^n$ for which $P(\varphi)=0$ for all $P\in \Sigma$. These comprise the set 
 compute the associated DA variety %$\text{Sol}(\Sigma)$ of $\Sigma$. {\color{red} This is the set 
 $\text{Sol}(\Sigma)=\{\varphi\in K_m ^n:P(\varphi)=0 \text{ for all }P\in\Sigma\}$. This is often hard; an easier subsidiary task is to compute its set of support series  $\text{sp}(\text{Sol}(\Sigma))\subset\mathbb{B}_m^n$, or its DA tropicalization. Explicitly, we will concentrate on understanding
\[
\text{sp}(\text{Sol}(\Sigma)):=\{a=(a_1,\ldots,a_n)\in\mathbb{B}_m^n\::\:a=\text{sp}(\varphi)\text{ for some }\varphi\in\text{Sol}(\Sigma)\}
\]
where $\text{sp}:K_m ^n\rightarrow\mathbb{B}_m ^n$ extends $\text{sp}:K_m \rightarrow\mathbb{B}_m $ coordinate-wise.

To define the DA tropicalization of $\text{Sol}(\Sigma)$, we'll make use of the following notion.

 %The fundamental theorem works with the extra differential information from this particular setting, so the following concepts will feature prominently hereafter.
 
\begin{dfn}
{An ideal of the differential ring $(R,D)$ is} \emph{differential} {whenever it is closed under the action of $D$.}
\end{dfn}
Now say that $(R,D)$ is a differential ring, and that $X$ is a subset of $R$. Following the convention adopted in \cite{HG19}, we let $(X)$ and $[X]$ denote the algebraic and differential ideals generated by $X$, respectively. By definition, $[X]$ is the intersection of all differential ideals of $R$ that contain $X$; so {\it $[X]$ is the algebraic ideal of $R=K_{m,n}$ spanned by the set $\{\Theta_{R}(J)(P):\:P\in X,\:J\in\mathbb{N}^m\}$}.

\begin{rem}\label{rem_sum_not_value}
{Given a semiring ${\rm S}$ and an element $P\in {\rm S}_{m,n}$, Remark~\ref{commutation} establishes that $\tfrac{\partial }{\partial t_i}(P(a))=\tfrac{\partial P}{\partial t_i}(a)$ for every %$\sum_Ma_ME_M=
   $a\in {\rm S}_m ^n$ and $i\in\{1,\ldots,m\}$. Now assume that ${\rm S}=K$; then $P(a)=0$ implies that $\tfrac{\partial P}{\partial t_i}(a)=0$; in particular,  %the differential ideal $[\Sigma]\subset K_{m,n}$ determined by 
a system $\Sigma\subseteq K_{m,n}$ and its associated differential ideal $[\Sigma]$  define the same DA varieties.}

%\medskip
{We saw in Proposition \ref{prop_reformulation_solution} that tropical vanishing depends  on the %expression of the evaluation 
realization of $P(a)$ as a sum of monomials and not on the {\it value} of $P(a)$. For example, when $P=x_{1,0}+x_{(0,1)}+(t^2+u^2)\in \mathbb{B}_{2,1}$, $a=t^2u+u^3\in\mathbb{B}_2$ is a solution of $P$, but not of $\tfrac{\partial P}{\partial t}$.}
%So, when $S=\mathbb{B}$, solutions depend on the expression $P(a)=\bigoplus_M a_M E_M(a)$ and not on the value $P(a)=\sum_{I\in A}T^I$.}
\end{rem}

\begin{dfn}
\label{Definition_TDAV}
Given $\Sigma\subseteq {K}_{m,n}$, %let $\text{Sol}(\Sigma)$ be the associated variety. 
the tropicalization {$\text{Sol}(\text{sp}([\Sigma]))$ of the DA variety  $\text{Sol}(\Sigma)$ is 
\[
   \text{Sol}(\text{sp}([\Sigma])):= \bigcap_{p\in \text{sp}([\Sigma])}\text{Sol}(p)
\]}
where $\text{sp}:K_{m,n}\rightarrow \mathbb{B}_{m,n}$ is the norm of Proposition~\ref{Longue_prop}.
%the differential ideal $[\Sigma]$.}
\end{dfn}

\begin{ex}[Tropical DA hypersurfaces]\label{ex_trop_hyp}
{Given $P\in K_{m,n}$, to compute the tropicalization of the associated hypersurface $\text{Sol}(P)=\{\varphi\in K_m ^n:P(\varphi)=0\}$ we first compute the differential ideal $[P]\subset K_{m,n}$; the tropicalization of $\text{Sol}(P)$ is then $\bigcap_{p\in\text{sp}([P])}\text{Sol}(p)$. }
\end{ex}

The following construction from \cite{FGH20+} is operative in the fundamental theorem, and generalizes the corresponding construction in the $m=1$ case from \cite{HG19}.

\begin{dfn}
\label{initial_ideal_def}
{Given $w\in\mathbb{B}_m ^n$, the }initial ideal {$\text{in}_w^*(G)$ of a differential ideal $G\subset {K}_{m,n}$ is the algebraic ideal generated by the initial forms $\{\text{in}_w^*(P):P\in G\}$ in ${K}_{m,n}$ defined in equation \eqref{initial_modified}.}
\end{dfn}

We now have all of the ingredients required to state the tropical fundamental theorem \cite[Thm. 9.6]{FGH20+} for differential algebraic geometry.

\begin{thm}[Fundamental theorem of tropical differential algebraic geometry] 
\label{EFT}
Fix $m,n\geq1$, let $K$ be an uncountable, algebraically closed field of characteristic zero, and let $[\Sigma]$ and $\text{Sol}(\Sigma)$ be the differential ideal and the DA variety associated to $\Sigma\subset K_{m,n}$ respectively. Then the following subsets of $\mathbb{B}_m^n$ coincide:
\begin{enumerate}
\item the tropicalization $\text{Sol}(\text{sp}([\Sigma])) = \bigcap_{p\in \text{sp}([\Sigma])}\text{Sol}(p)$ of   $\text{Sol}(\Sigma)$ as in Definition \ref{Definition_TDAV};
\item the set $\{w\in \mathbb{B}_m ^n\::\:\text{in}_w^*([\Sigma])\text{ is monomial-free}\}$ as in Definition \ref{initial_ideal_def}; and
\item the set $\text{sp}(\text{Sol}(\Sigma))=\{\text{sp}(\varphi)\::\:\varphi\in\text{Sol}(\Sigma)\}$ of coordinatewise supports of points in $\text{Sol}(\Sigma)$.
\end{enumerate}
\end{thm} 

The tropical fundamental theorem \ref{EFT} characterizes tropical DA varieties as (1) sets of formal $\mathbb{B}$-power series solutions of the tropical differential system $\text{sp}([\Sigma])$ associated to a differential ideal $[\Sigma]\subset K_{m,n}$; (2) weight vectors that single out monomial-free initial ideals of a fixed differential ideal $[\Sigma]$; and (3) as support series of all formal ${K}$-power series solutions $\text{Sol}(\Sigma)$ of a system $\Sigma\subset K_{m,n}$. It establishes, in particular, that any formal $\mathbb{B}$-power series solution of the tropical system $\text{sp}([\Sigma])$ lifts to a formal ${K}$-power series solution for $\Sigma$.

\begin{ex}
\label{ex_trop_hyp_2}
{%Recall from Example \ref{ex_trop_hyp} that the tropical DA hypersurface associated to $P\in K_{m,n}$ was $X:=\bigcap_{p\in\text{sp}([P])}\text{Sol}(p)$. By 
Applying Theorem \ref{EFT}, we see that the tropical DA hypersurface associated to $P\in K_{m,n}$ is $X= \{\text{sp}(\varphi):\:\varphi\in K_m ^n,\: P(\varphi)=0\}$; cf. Example~\ref{ex_trop_hyp}.}
\end{ex}

\begin{rem}
{Important examples of base fields ${K}$ to which our theory applies are $\mathbb{C}$ and $\mathbb{Q}_p^{\text{alg}}$, the algebraic closure of the field of $p$-adic numbers. Each of these fields is equipped with a natural topology which comes from an absolute value, which is archimedean in the complex case ($\mb{C}_\infty$), and non-archimedean in the $p$-adic case; let $\mathbb{C}_p$ denote the metric completion of $\mathbb{Q}_p^{\text{alg}}$. %Their completions $\mathbb{C}$ and $\mathbb{C}_p:=\overline{\mathbb{Q}_p^{\text{alg}}}$ come equipped with natural absolute values with respect to which they are complete, 
It is natural to ask for conditions under which the solutions $\text{Sol}(\Sigma)$ are tuples of analytic functions in a neighborhood of the origin with respect to the canonical metrics inherited from $\mb{C}_\infty$ and $\mb{C}_p$, respectively.}
%(archimedean in the complex case, or non-archimedean in the $p$-adic case).}
\end{rem}

\begin{rem}
\label{rem_paradigms}
{According to \cite{BH19}, the solution set of any system $\Sigma\subset K_{m,n}$ is equal to the solution set of a \emph{finite} system $\{P_1,\ldots,P_s\}\subset \Sigma$. Each polynomial $P_i$ imposes an infinite number of conditions on the coefficients of the solutions $\varphi\in\text{Sol}(\Sigma)$.}

%\medskip
{
\hspace{-3pt}On the other hand, any given polynomial $Q\in\mathbb{B}_{m,n}$ imposes only a finite number of conditions on the supports of the solutions $\varphi\in\text{Sol}(Q)$, and in view of 
\[
\text{sp}(\text{Sol}(P_1,\ldots,P_s))=\text{sp}(\text{Sol}(\Sigma))=\text{Sol}(\text{sp}([\Sigma]))
\]
we %expect 
require an infinite number of elements $Q\in \text{sp}([\Sigma])$ in order to describe this common set. %Thus we have the following two paradigms about these two types of theories of algebraic differential equations:
The upshot is the following dichotomy between classical and tropical theories of algebraic differential equations:
 \begin{itemize}
     \item the classical theory requires solving a {\it finite number of infinite systems} of algebraic equations, while
     \item the tropical theory requires solving an {\it infinite number of finite systems} of boolean equations.
 \end{itemize}}
\end{rem}

In Theorem~\ref{EEFT} of the next section, we will extend items (2) and (3) of the fundamental theorem to the fraction field $K(\!(T)\!)$ of $K[\![T]\!] $.

\section{From $K[\![T]\!] $ to $K(\!(T)\!)$}
\label{Section_Fromringtofield}

Paraphrasing \cite{HG19}, {\it it is more convenient to develop Groebner basis theory over fields than rings.} With this principle in mind, in this section we generalize the valuative results obtained for the differential ring $K_m=(K[\![T]\!] ,D)$ in section~\ref{Section_QVIF} to the differential field $F_m:=(K(\!(T)\!),D)$. %This part can be seen as an application of Proposition \ref{Prop_MC}, this is, that  $V\mathbb{B}[T]$ is integral (hence it has no zero divisors). 
The results in this section also extend those of \cite[Sect. 2.2]{HG19} to the $m>1$ case.
The key ingredient in all of this is the following property of the semiring $V\mathbb{B}[T]$ of vertex polynomials.

\begin{prop}
\label{Prop_MC}
The semiring $V\mathbb{B}[T]$ is {integral} (or {multiplicatively cancellative}); that is, if we have $a\odot c=b\odot c$, then  $c\neq0$ or $a=b$.
\end{prop}
\begin{proof}
Let $a,b\in V\mathbb{B}[T]$ be nonzero elements with supports $A$ and $B$ respectively. We will show first that for any $I\in A$, there exists $J\in B$ such that $I+J\in a\odot b$. Indeed, intersecting the normal fans of $\text{New}(A)$ and $\text{New}(B)$ results in the normal cone $C_I$ of each vertex $I$ of $\text{New}(A)$ being subdivided by cones of $\text{New}(B)$.

%{\color{red} Let $a,b\in V\mathbb{B}[T]$ different from zero and with supports $A$ and $B$ respectively. We will show first that for any $I\in A$, there exists $J\in B$ such that $I+J\in a\odot b$. We consider the intersections of the normal fans of $\text{New}(A)$ and $\text{New}(B)$, then the normal cone $C_I$ of  each vertex $I$ of $\text{New}(A)$ gets subdivided by cones of $\text{New}(B)$.

Since both normal fans are maximal dimensional, there is necessarily some cone $C_J$ of $\text{New}(B)$ that contains or has non-empty intersection with $C_I$, so $I+J$ is a vertex of $\text{New}(A)+\text{New}(B)$, i.e., $I+J\in a\odot b$.

Now suppose $a\odot c=b\odot c$, where  $c\neq0$ has support $C$; we will show that $A=B$. To do so we argue indirectly, assuming for the sake of argument that $A \setminus B \neq \emptyset$. Say $K\in A\setminus B$, and that $C=\{J_1,\ldots,J_s\}$. As $a\odot c=b\odot c$, no $K+J_i$, $i=1,\dots,s$ belongs to $b\odot c$; but this contradicts the fact that $K+J_i\in a\odot c$ for some $J_i$. Thus $V(a) \subset V(b)$; by symmetry, we conclude that $V(a)=V(b)$.
\end{proof}
\begin{rem}
Let $\sim$ denote the semiring congruence  on $\mathbb{B}[\![T]\!]$ consisting of pairs $(a,b)$ of boolean power series for which $ac=bc$ for some $c \neq 0$. Suppose that $a\sim b$; as $V$ is a homomorphism, we then have $V(ac)=V(a)\odot V(c)=V(b)\odot V(c)=V(bc)$, and it follows from \eqref{Prop_MC} that $V(a)=V(b)$.

The semiring $\faktor{\mathbb{B}[\![T]\!] }{\sim}$ is cancellative by definition, and we have a commutative diagram of surjective maps of semirings
\begin{equation*}
    \xymatrix{\mathbb{B}[\![T]\!] \ar[d]_\pi\ar[dr]^V&\\
    \faktor{\mathbb{B}[\![T]\!] }{\sim}\ar[r]_{\widetilde{V}}&V\mathbb{B}[T].}
\end{equation*}
\end{rem}

%Since $V\mathbb{B}[T]$ is free of zero divisors, 
As $V\mathbb{B}[T]$ has no nontrivial zero divisors, there is a well-defined semifield $V\mathbb{B}(T)=\text{Frac}(V\mathbb{B}[T])$ of fractions $\frac{a}{b}$ with $a,b\in  V\mathbb{B}[T]$, $b\neq0$. The product and sum of fractions in $V\mathbb{B}(T)$ are defined as usual by $\frac{a}{b}\odot\frac{c}{d}=\frac{a\odot c}{b\odot d}$, and  $\frac{a}{b}\oplus\frac{c}{d}=\frac{a\odot d\oplus b\odot c}{b\odot d}$.

The upshot of Proposition~\ref{Prop_MC} is that  $\frac{a}{b}=\frac{c}{d}$ in $V\mathbb{B}(T)$ if and only if $a\odot d=b\odot c$; furthermore, the map $ V\mathbb{B}[T]\ra V\mathbb{B}(T)$ that sends $a$ to $\frac{a}{1}$ is an embedding.  For a detailed account, see \cite[Sect. 11]{JG}.

 \begin{coro}
 \label{extension_trop}
  The map $\text{trop}:K(\!(T)\!) \rightarrow V\mathbb{B}(T)$ defined by $\trop(\frac{f}{g}):=\frac{\trop(f)}{\trop(g)}$ is a  $K$-algebra valuation.
 \end{coro}
 
 \begin{proof}
 Follows immediately from Theorem~\ref{VBT_theorem}, which establishes that $\text{trop}:K[\![T]\!] \rightarrow V\mathbb{B}[T]$ has the same property.
 \end{proof}
 
 Now let $F_{m,n}:=(K(\!(T)\!) [x_{i,J}: i,J],D)$ denote the  differential ring of differential polynomials with coefficients in the differential field $K(\!(T)\!) $. By clearing denominators, every $P\in F_{m,n}$ decomposes as
 \begin{equation}
 \label{from_field_to_ring}
     \lambda P=Q, \text{ where }\lambda\in K[\![T]\!] \setminus\{0\}\text{ and }Q\in K_{m,n}.
 \end{equation}
 The decomposition \eqref{from_field_to_ring} is the key to extending our results from $K[\![T]\!]$ to $K(\!(T)\!)$.
 
  \begin{coro}
  Given $w\in\mathbb{B}_m ^n$, the map $\trop_w:K(\!(T)\!) [x_{i,J}: i,J] \ra V\mathbb{B}(T)$ that sends $P=\sum_Ma_ME_M$ to $\trop_w(P):=\bigoplus_M(\trop(a_M)\odot V(E_M(w)))$ is a multiplicative $K$-algebra  seminorm.
 \end{coro}
 \begin{proof}
 Given $P\in F_{m,n}$, write $P=\frac{Q}{\lambda}$ as in \eqref{from_field_to_ring}, so that $\trop_w(P)=\frac{\trop_w(Q)}{\trop(\lambda)}$; the desired result follows from Corollary \ref{coro_Longue_prop}, which establishes that $\trop_w:K[\![T]\!] [x_{i,J}: i,J]\ra V\mathbb{B}[T]$ is a multiplicative $K$-algebra  seminorm. 
 \end{proof}
 
 We propose the following generalization of \eqref{initial_modified} from $K_{m,n}$ to $F_{m,n}$. 
 
 \begin{dfn}
  \label{initial_field}
Given $P\in F_{m,n}$, write $P=\frac{Q}{\lambda}$ as in \eqref{from_field_to_ring}. The {\it initial form} $\text{in}_w^*(P)$ of $P$ at $w \in \mb{B}_m^n$ is
 $\text{in}_w^*(P):=\frac{\text{in}_w^*(Q)}{\overline{\lambda}}.$ 
 \end{dfn}
 
 %{\color{red}
\begin{rem}
The notion of relevancy may be extended from $V\mathbb{B}[T]$ to $V\mathbb{B}(T)$ as follows: whenever $\frac{a}{b}\leq\frac{c}{d}$, we say that $\frac{a}{b}$ is relevant for $\frac{c}{d}$ provided that $a\odot d\nprec c\odot b$ in $V\mathbb{B}[T]$. %In the expression $\lambda P=Q$ for $P=\sum_{M}\frac{a_M}{b_M}E_M$, for a monomial  $c_ME_M$ of $Q$ we have 
Now say as before that $\lambda P=Q$, with $P=\sum_{M}\frac{a_M}{b_M}E_M$. Given a monomial $c_ME_M$ of $Q$, we have
$$\text{trop}_w(\frac{a_M\la}{b_M}E_M)=\text{trop}_w(c_ME_M)\nprec \text{trop}_w(Q)=\text{trop}(\lambda)\odot \text{trop}_w(P).$$
It follows from Remark \ref{properties_of_irrelevant} that $\text{in}_w^*(Q)$ (and hence $\text{in}_w^*(P)$) is supported on the set of $M$ for which $\text{trop}_w(\frac{a_M}{b_M}E_M)\nprec \text{trop}_w(P)$. %$\{M\::\:\text{trop}_w(\frac{a_M}{b_M}E_M)\nprec \text{trop}_w(P)\}$.
\end{rem}
%}
 
 We thus have a decomposition 
\begin{equation}\label{decomposition_relevant}
    \lambda P=Q=\text{in}_w^{\ast}(Q)+Q^{\ast}=\overline{\lambda}\cdot \text{in}_w^*(P)+P^{\pr}
\end{equation}
in which $\trop_w(b_NE_N)\prec \trop_w(\lambda P)$ for every monomial $b_NE_N$ of $P^{\pr}$. In fact, more is true; namely, $\trop_w(P^{\pr})\prec \trop_w(\lambda P)=\trop_w(Q)=\trop_w(\text{in}_w^{\ast}(Q))$. So again $P^{\pr}=Q-\text{in}_w^{\ast}(Q)$ is the  $w-$irrelevant part of $\lambda P$.

We close this preliminary subsection with some significant extensions of previous results. It is worth noting that in the extended version of the fundamental theorem \ref{EEFT}, the supports of solutions in $K[\![T]\!]^n$ for systems defined over the field $K(\!(T)\!)$ are controlled by the theory over $K[\![T]\!]$.

\begin{thm}
\label{Previous_EEFT}
 A differential ideal $G \subset F_{m,n}$ has a solution in $K_m ^n$ with support vector $w\in \mathbb{B}_m^n$ if and only if $\{\text{in}_w^*(P):P\in G\}$ contains no differential monomials.
 \end{thm}
 
 \begin{proof}
 The proof imitates that of \cite[Cor. 1]{HG19}, 
 %but some details must be checked.
 as follows.
 Let $G_0=G\cap K_{m,n}$; then for any $P\in G$ equation \eqref{from_field_to_ring} yields $Q=\lambda P\in G_0$, and  $\text{in}_w^*(Q)=\overline{\lambda} \cdot \text{in}_w^*(P)$ by definition. On the other hand, clearly $\varphi\in K_m ^n$ is a solution of $G$ if and only if it is a solution of $G_0$; the desired result now follows from part 2 of Theorem~\ref{EFT}.
 \end{proof}
 
  \begin{dfn}\label{initial_ideal_ext_def} 
{Given a differential ideal $G \subset F_{m,n}$ and $w\in \mathbb{B}_m^n$, the \emph{initial ideal} $\text{in}_w^*(G)$ of $G$ with respect to $w$ is the algebraic ideal $(\text{in}_w^*(P):P\in G)$ of $F_{m,n}$.}
 \end{dfn}

 \begin{thm}[Fundamental theorem of tropical differential algebraic geometry over $F_m$] 
\label{EEFT}
Fix $m,n\geq1$, let $K$ be an uncountable, algebraically closed field of characteristic zero, and let $\Sigma\subset F_{m,n}$ be a family of  differential polynomials with associated differential ideal $[\Sigma]$ and solution set $\text{Sol}(\Sigma)=\{\varphi\in K_m ^n:P(\varphi)=0 \text{ for all }P\in\Sigma\}$. The following subsets of $\mathbb{B}_m ^n$ coincide:
\begin{enumerate}
%\item the tropicalization $\text{Sol}(\text{Supp}([\Sigma]))$ of the DA variety  $\text{Sol}(\Sigma)$ as in Definition \ref{Definition_TDAV};
\item the set $\{w\in \mathbb{B}_m ^n:\text{in}_w^*([\Sigma])\text{ is monomial-free}\}$ as in Definition \ref{initial_ideal_ext_def}; and
\item the set $\text{sp}(\text{Sol}(\Sigma))=\{\text{sp}(\varphi):\varphi\in\text{Sol}(\Sigma)\}$ of coordinatewise supports of points in   $\text{Sol}(\Sigma)$.
\end{enumerate}
\end{thm} 

\begin{proof}
 The proof is similar to that of \cite[Lem. 3]{HG19}. Letting $G=[\Sigma]$, it suffices to show that if $\text{in}_w^*(G)$ contains a  monomial, then $\{\text{in}_w^*(P):\:P\in G\}$ also contains a  monomial (as the converse is clear). We then conclude by applying Theorem~\ref{Previous_EEFT}.
 
 %\medskip
 Accordingly, say that $E=\sum_iZ_i\text{in}_w^*(P_i)$, where $P_i\in G$ and $Z_i\in F_{m,n}$ satisfy $\lambda_iP_i=R_i$ and $\mu_iZ_i=S_i$ for some $R_i,S_i\in K_{m,n}$ and $\lambda_i,\mu_i\in K[\![T]\!]$. According to Definition~\ref{initial_field}, we have $\text{in}_w^*(P_i)=\frac{\text{in}_w^*(R_i)}{\overline{\lambda_i}}$; and thus $cE=\sum_i(c_iS_i)\text{in}_w^*(R_i)$, where $c=\prod_j\mu_j\overline{\lambda_j}$ and $c_i=\prod_{j\neq i}\mu_j\overline{\lambda_j}$.
 
 %{\color{red} 
 Let $P=\sum_i(cS_i)R_i=\sum_i(cS_i)\lambda_iP_i\in G$; we will show that $\overline{c}E=\text{in}_w^*(P)$. Indeed, for each index $i$, write $c_iS_i=\sum_jb_{N(ij)}E_{N(ij)}$ and $R_i=\sum_ka_{M(ik)}E_{M(ik)}$, so that $P=\sum_i\biggl[\sum_{j,k}b_{N(ij)}a_{M(ik)}E_{N(ij)}E_{M(ik)}\biggr]$ and
\[
     %\begin{aligned}
     %P=\sum_i\biggl[\sum_{j,k}b_{N(ij)}a_{M(ik)}E_{N(ij)}E_{M(ik)}\biggr]\text{ and } 
     \trop_w(P):=\bigoplus_i\biggl[\bigoplus_{j,k}\trop_w(b_{N(ij)}a_{M(ik)}E_{N(ij)}E_{M(ik)})\biggr].
     %\end{aligned}
\]

Further write $a_{M(ik)}=\overline{a_{M(ik)}}+a_{M(ik)}^*$ if $M(ik)$ is in the support of $\text{in}_w^*(R_i)$, and $a_{M(ik)}=a_{M(ik)}^*$ otherwise.  For every $i$ we have $\trop_w(R_i^*)\prec \trop_w(R_i)=\trop_w(\text{in}_w^*(R_i))$, so  $\trop_w(cS_iR_i^*)\prec \trop_w(cS_iR_i)=\trop_w(cS_i\text{in}_w^*(R_i))$ by Remark~\ref{properties_of_irrelevant}.  Hence 
 $$\bigoplus_{j,k}\trop_w(b_{N(ij)}a_{M(ik)}E_{N(ij)}E_{M(ik)})=\bigoplus_{j,k}\trop_w(b_{N(ij)}\overline{a_{M(ik)}}E_{N(ij)}E_{M(ik)})$$
 and summing over $i$ yields
 $$\trop_w(P):=\bigoplus_i\biggl[\bigoplus_{j,k}\trop_w(b_{N(ij)}\overline{a_{M(ik)}}E_{N(ij)}E_{M(ik)})\biggr].$$
 
 Thus the only monomials of $P$ in (the support of) $\text{in}_w^*(P)$ are those in 
 $cE=\sum_i(c_iS_i)\text{in}_w^*(R_i)$. Suppose $\text{trop}_w(b_{N(ij)}\overline{a_{M(ik)}}E_{N(ij)}E_{M(ik)})\nprec \text{in}_w^*(P)$ for one of these. Then either i) $E_{N(ij)}E_{M(ik)}=E$ or ii) $E_{N(ij)}E_{M(ik)}\neq E$,
% We have two cases: 
% \begin{enumerate}
%     \item $E_{N(ij)}E_{M(ik)}=E$,
%     \item $E_{N(ij)}E_{M(ik)}\neq E$
% \end{enumerate}
 where the sum over all the monomials of the second type is zero by hypothesis. %If the preceding monomial belongs to the second family, then we conclude that its additive inverse 
 In the second case, we conclude that
 $-b_{N(ij)}\overline{a_{M(ik)}}E_{N(ij)}E_{M(ik)}$ also appears in $cE$, and their initial forms cancel.  Thus $\text{in}_w^*(P)=\overline{c}E$, and our proof is complete.
 \end{proof}
 
 Given a differential ideal $I\subset F_{m,n}$ and a weight vector $w\in\mathbb{B}_m^n$, a set $\mathcal{G}\subset I$ is a {\it tropical basis} for $I$ with respect to $w$ provided $\text{in}_w^*(I)=(\text{in}_w^*(\Theta(J)(g)):\:g\in\mathcal{G},J\in \mathbb{N}^m)$. The following result generalizes \cite[Prop. 1]{HG19}; our strategy of proof is slightly different.

 \begin{prop}
 \label{prop_local_trop_basis}
 Let $\omega_m\in \mathbb{B}_m^n$ denote the unique series vector with support $\mathbb{N}^m$. Given $G\subset K_{0,n}$, let $I=[G]$ denote the differential ideal of $F_{m,n}$ that it generates. Then $G$ is a tropical basis for $I$ with respect to $\omega_m$.
 \end{prop}
 
 \begin{proof}
 Set $\mathcal{G}:=\{\Theta(J)(g):\:g\in G,\: J\in\mathbb{N}^m\}$ and $w:=\omega_m$. As $\mathbb{N}^m\subset \text{Supp}(w)$ and $G\subset K_{0,n}$, we have $\text{in}_{w}^*(h)=h$ for every $h\in\mathcal{G}$.
 On the other hand, given any $f\in I$, there is some polynomial $f_1\in K_{m,n}$ and $\lambda\in K$ for which $\lambda f=f_1=\sum_ia_ih_i$, where $a_i\in K_{m,n}$ and $h_i\in\mathcal{G}$.
 
 %\medskip
 We will compute $\text{in}_w^*(f_1)$. To this end, we decompose $a_i$ as $a_i=\text{in}_w^*(a_i)+a^{\pr}_i$, so that $a_ih_i=\text{in}_w^*(a_i)h_i+a^{\pr}_i$. For every monomial $b_NE_M$ of $h_i$, we have $v_w(h_i)=v_w(b_NE_M)=1$. It follows that $\text{in}_w^*(a_ih_i)=\text{in}_w^*(a_i)h_i$, and consequently that
 \[
 \text{in}_w^*(f_1)=\sum_i\text{in}_w^*(a_i)h_i \text{ and } \text{in}_w^*(f_1)=\frac{\text{in}_w^*(f_1)}{\overline{\lambda}}
 \]
 where the sum is over all relevant indices $i$, which means precisely that $G$ is a tropical basis for $I$ with respect to $w$.
 \end{proof}
 
 \subsection{Initial degenerations}\label{initial_degenerations}
%We will now switch to algebraic initial degeneration. 
The geometric counterparts of initial forms are initial {\it degenerations}. Each of these is specified by a scheme-theoretic {\it model} that is flat over the spectrum of the {\it valuation ring} $K(\!(T)\!)^\circ=\{x\in K(\!(T)\!):\trop(x)\leq1\}$ corresponding to the valuation $\trop:K(\!(T)\!)\to V\mathbb{B}(T)$, as in \cite{gub}.  

%\medskip
Given an affine scheme $X$ defined by an ideal $I\subset K(\!(T)\!)[x_{i,J}:i,J]$, we now let $\mathcal{X}$ denote the spectrum of a quotient ring $K(\!(T)\!)^\circ[x_{i,J}:i,J]/I^{\pr}$, where $I^{\pr}$ is a translation of $I$ induced by a weight vector $w\in\mathbb{B}_m ^n$: this is our putative model. We will show that, for every $w\in\mathbb{B}_m^n$, an analogue of the translation map $K(\!(T)\!) [x_{i,J}:i,J] \ra K(\!(T)\!) ^\circ[x_{i,J}:i,J]$ %\eqref{translation_map} 
defined when $m=1$ in \cite{FT20} also exists when $m>1$. On the other hand, we will also show that when $m>1$, there is a tower of ring extensions $K[\![T]\!] \subsetneq K(\!(T)\!)^\circ\subset K(\!(T)\!)$ which diverges from the $m=1$ case. In particular, $K(\!(T)\!)^\circ$ is not a local ring, and therefore is not a valuation ring in the traditional sense; nor is it closed under taking partial derivatives.
%{\color{red}It is also not closed under the partial derivatives.}

 \begin{coro}
 %Consider $\trop:K(\!(T)\!) \longrightarrow V\mathbb{B}(T)$. Then
 Let $\mathcal{C}:=\{x\in K(\!(T)\!):\frac{\partial x}{\partial t_j}=0 \text{ for every } j\}$. Then $\mc{C}=K$, while %$K(\!(T)\!)^{\circ}=\{x\in K(\!(T)\!):\:\trop(x)\leq 1 \}$ 
 $K(\!(T)\!)^{\circ}$ is a subring of $K(\!(T)\!)$.
 %, where $\trop:K(\!(T)\!) \ra V\mathbb{B}(T)$ is the extended tropicalization map.
 \end{coro}
 
 \begin{proof}
 The first claim is clear; the second follows from the fact that $V\mathbb{B}(T)^{\circ}:=\{a\in V\mathbb{B}(T): a\leq1\}$ is a subsemiring of $V\mathbb{B}(T)$, together with Corollary~\ref{extension_trop}.
 \end{proof}
 
 \begin{dfn}
   {We call $K(\!(T)\!) ^{\circ}$ (resp., $V\mathbb{B}(T)^{\circ}$) the {\it ring of integers} of $K(\!(T)\!) $ (resp., the {\it semiring of integers} of $V\mathbb{B}(T)$).}
  \end{dfn}
 
  \begin{prop}
  \label{ring_of_integers}
  The ring $K(\!(T)\!)^{\circ}$ is an extension of $K[\![T]\!]$, and it is local and differential if and only if $m=1$.
  \end{prop}
  
    \begin{proof}
  If $m=1$, then $K(\!(T)\!) ^\circ=K[\![T]\!]$ and there is nothing to prove, so we shall suppose that $m>1$. Since every $a\in V\mathbb{B}[T]$ is less than  $1$, we have
   $V\mathbb{B}[T]\subset V\mathbb{B}(T)^{\circ}$. In particular, since $\trop(K[\![T]\!] )=V\mathbb{B}[T]$, we deduce that $K[\![T]\!] $ is a subring of $K(\!(T)\!)^{\circ}$. In general this inclusion is proper; for example, we have $\frac{t_1 t_2}{t_1+t_2}\in K(\!(T)\!)^{\circ}\setminus K[\![T]\!]$.
   
   To prove the second point, we use the fact that a ring is local if and only if the sum of any two non units is a non-unit, and we produce an explicit unit that decomposes as a sum of non-units when $m=2$.
%will give a counterexample of this fact for $m=2$.
Appealing to Lemma \ref{order_relation} again, we have $\frac{a}{b}<1$ in $V\mathbb{B}(T)$ if and only if $A=A^{\pr} \sqcup B^{\pr}$, %(disjoint union), 
with $A^{\pr}\subsetneq B$ and $B^{\pr}\subset (\widetilde{B}\setminus B)$. Thus $x=\frac{t_1}{t_1+t_2}$ and $y=\frac{t_2}{t_1+t_2}$ are non-units in $K(\!(T)\!) ^{\circ}$ for which $x+y=1$.

Finally, the ring $K(\!(T)\!)^\circ$ is not closed under differentiation; for example, when $m=2$, the element $\frac{t_1}{t_1+t_2}$ belongs to $K(\!(T)\!)^\circ$, while $\frac{\partial}{\partial t_1}(\frac{t_1}{t_1+t_2})=\frac{t_2}{(t_1+t_2)^2}$ does not.
  \end{proof}
  
  The following construction is inspired by one carried out in \cite{FT20} in the $m=1$ case.
%in the differential ordinary case $m=1$. 
For each weight vector $w=(w_1,\ldots,w_n)\in\mathbb{B}_m^n$, we define a map 
    \begin{equation}
    \label{translation_map}
    \begin{aligned}
    K(\!(T)\!) [x_{i,J}:i,J]&\longrightarrow K(\!(T)\!) ^\circ[x_{i,J}:i,J]\\
    P&\mapsto P_w.
    \end{aligned}
    \end{equation}

For every $1\leq i\leq m$ and $J\in \mathbb{N}^m$, we let $T(w_i,J)\in K[\![T]\!] $ denote the series $V(\Theta_{\mathbb{B}_m}(J)w_i)$. In terms of  the canonical section $e_K:\mathcal{P}(\mathbb{N}^m)\to K[\![T]\!]$, we have 
\[
T(w_i,J)=e_K(\text{Vert}((W_i-J)_{\geq0}))
\]
where $W_i=\text{Supp}(w_i)$ for $i=1,\ldots,n$. Given $P\in F_{m,n}$, if $\trop_w(P)=\frac{a}{b}\neq0$ with $A=\text{Supp}(a)$ and $B=\text{Supp}(b)$; we then set $T(\trop_w(P))^{-1}:=\frac{e_K(B)}{e_K(A)} \in K(\!(T)\!)$.

%\medskip
We now specify a differential polynomial $P_w$ by substituting every instance of $x_{i,J}$ in $P$ by $T(w_i,J)x_{i,J}$  and then multiplying the result by $T(\trop_w(P))^{-1}$:

\begin{equation*}
P_w(x_{i,J}):=\begin{cases}
  T(\trop_w(P))^{-1}P(T(w_i,J)x_{i,J}),&\text{ if }\trop_w(P)\neq0\\
  
 0,&\text{ if }\trop_w(P)=0. \\
\end{cases}
\end{equation*}

\begin{lemma}
\label{toric_degeneration}
The polynomial $P_w$ has coefficients in  $K(\!(T)\!)^{\circ}$.
\end{lemma}
\begin{proof}
Let $P=\sum_Ma_ME_M$. We then have 
$P_w(x)= %T(\trop_w(P))^{-1}P(T(w_i,J)x_{i,J})=T(\trop_w(P))^{-1}\sum_Ma_M\prod_{i,J}(T(w_i,J)x_{i,J})^{M_{i,J}}=
\sum_Mb_ME_M$, where $b_M=T(\trop_w(P))^{-1}a_M\prod_{i,J}(T(w_i,J))^{M_{i,J}}.$
We need to show that $\trop(b_M)\leq1$, i.e., that 
\[
\trop(a_M)\odot V(E_M(w))=
\trop \bigg(a_M\prod_{i,J}(T(w_i,J))^{M_{i,J}}\bigg)\leq\trop_w(P).
\]
Here the equality is clear, while the inequality holds because $\trop_w(P)$ is the least upper bound of $\bigoplus_M\trop_w(a_ME_M)$.
\end{proof}

Whenever $G\subset K(\!(T)\!) [x_{i,J}:i,J]$ is a differential ideal and the ring 
\begin{equation}
    \label{ring_for_model}
    R_w:=K(\!(T)\!) ^\circ[x_{i,J}:i,J]/(P_w:\:P\in G)
\end{equation}
is flat over $K(\!(T)\!)^\circ$,
there is an associated {\it (algebraic) initial degeneration} of the spectrum of $R=K(\!(T)\!) [x_{i,J}:i,J]/G$ given by specializing the scheme $\mathcal{X}(w):=\text{Spec}(R_w)$
%\text{Spec}(K(\!(T)\!) ^\circ[x_{i,J}:1\leq i\leq n,\:J\in\mathbb{N}^m]/(P_w:\:P\in G))
to the fiber over some closed point $p\in \text{Spec}(K(\!(T)\!)^\circ)$.

\section{Computational aspects and open problems}
 \label{Section_CA}
 In this section we discuss issues at play in effectively computing tropical DA varieties, and we highlight several outstanding problems in need of resolution.
 
 The first issue concerns the finiteness of presentations of differential ideals, and tropical bases. As stated, the fundamental theorems \ref{EFT} and \ref{EEFT} are about  differential ideals $G$, %$G\subset {K}_{m,n}$ of $G\subset {F}_{m,n}$, 
which are infinite sets of differential polynomials. The Ritt-Raudenbush basis theorem (see \cite{BH19}) establishes that there are always finitely many $P_1,\ldots,P_s\in G$ for which $\text{Sol}(G)=\text{Sol}(P_1)\cap \cdots\cap \text{Sol}(P_s)$. 
It is not true in general, however, that the DA tropicalization of $\text{Sol}(G)$ is described by the supports of these generators; indeed, as explained in Remark~\ref{rem_paradigms}, an infinite set $\{\text{sp}(P):P\in G\}$ is generally required to recover the DA tropicalization.

A {\it tropical DA basis}
%\footnote{This should not be confused with the tropical bases for differential ideals equipped with a particular choice of weight vector $w$ studied in section~\ref{Section_Fromringtofield}.} 
for $G$ is a (preferably smaller) subsystem $\Phi\subset G$ from which we may compute both the DA variety $\text{Sol}(\Phi)=\text{Sol}(G)$ and its corresponding DA tropicalization $\text{Sol}(\text{sp}(\Phi))=\text{Sol}(\text{sp}(G))$. It is worth noting that Groebner bases exist in both the tropical and the differential settings, and that tropical differential Groebner bases for ordinary differential tropical systems were introduced in \cite{HG19}, but at present these are lacking in the more general partial differential case. These would simultaneously represent adaptations of tropical Groebner bases to the differential setting, and of differential Groebner bases to the tropical setting.

\noindent {\it Problem 1.} Define a proper notion of tropical basis for the partial differential setting.

%\medskip
The second issue is that of effectively evaluating vertex polynomials. Given $P= \sum_Ma_ME_M$ in $K_{m,n}$ and $a\in {K}[\![T]\!]^n$, we have $P(a)=\sum_Ma_ME_M(a)$, and consequently $V(P(a))=\bigoplus_MV(a_M E_M(a))=\bigoplus_MV(a_M)\odot V(E_M(a))
$, as $V:\mathbb{B}[\![T]\!] \ra V\mathbb{B}[T]$ is a semiring homomorphism. As a result, solutions of $P$ may be formulated  in terms of the vertex polynomials $V(a_M)\odot V(E_M(a))$ (as per the original definition given in \cite{FGH20}).

Using the decomposition \eqref{New_decomposition}, it should be possible to compute the vertex polynomials $V(a)$ algorithmically by combining 
\begin{enumerate}
\item a routine that computes the vertex set $V(P)$ of a polyhedron $P$; and
    \item a routine that computes the distinguished polytopal representative $P=P_a$ in the decomposition \eqref{newton_decomposition} of the integral polyhedron $\text{New}(a)$.
    %a subroutine that computes the vertex set $V(P_a)$ of an integral polyhedron $P_a$. 
\end{enumerate}
This is significant, as implementations of each of these routines are  available. %software available (PORTA, Normaliz, Sage, etc.)

A closely-related problem is that of producing concrete formulas for the binary operations $\oplus$ and $\odot$ on $V\mathbb{B}[T]$; this, in turn, relates to a number of subsidiary problems involving operations on polytopes. Recall from \eqref{uniqueness} given vertex polynomials $a$ and $b$ with supports $A$ and $B$ respectively, we have %the chains of inclusions from \eqref{uniqueness}:
\[
V(ab)\subseteq \mathcal{V}(\Delta_a\otimes \Delta_b)\subseteq A+B,\quad \text{ and }\quad V(a+b)\subseteq \mathcal{V}(\Delta_a\oplus \Delta_b)\subseteq A\cup B
\]
where $\Delta_{\al}=\text{Conv}(\al)$.
%If $a\in V\mathbb{B}[T]$, we denote by $\Delta_a=\text{Conv}(a)$, so that the set of vertices $\mathcal{V}(\Delta_a)$ of $\Delta_a$ is precisely $a$.  We have the following chains 
Thus, as pointed out in \cite{DT}, one strategy for effectively computing $a\odot b=V(ab)$ involves first computing $A+B$, and then applying a second algorithm to remove all points that are not in $\mathcal{V}(\Delta_a\otimes \Delta_b)$.
%We have $ \mathcal{V}(P_a+P_b)$ consists of the set of $x\in \mathcal{V}(P_a)+\mathcal{V}(P_b)$ such that the decomposition of $x$ as a sum of elements of $a$ and $b$ is unique. For a proof and algorithms for computing this, see 
On the other hand, computing $a \oplus b$ naturally leads to that of computing $\mathcal{V}(\Delta_a\oplus \Delta_b)$; the latter problem is referred to in the literature as {\it redundancy removal}.

\noindent {\it Problem 2.} Implement  algorithms for computing vertex polynomials and vertex sets.

%\medskip
Another issue we would like to highlight is that of convergence. Thus far, we have only considered solutions as vectors of formal power series in $K_m ^n$, but for geometric applications it would be useful to know whether they represent analytic functions $\Omega\rightarrow \mathbb{A}_K^{1,\text{an}}$ with a common domain $0\in\Omega\subset \mathbb{A}_K^{m,\text{an}}$, whenver $K$ is a Banach field. Possibilities are stratified according to whether or not the field $K$ is Archimedean.
\begin{enumerate}
\item[(1)] {\it The Archimedean case}. Here the only possibility is that ${K}=\mathbb{C}$ with the Euclidean topology; thus $\mathbb{A}_{{K}}^{m,\text{an}}=(\mathbb{C}_\infty^m,\mathcal{H})$, where $\mathcal{H}$ is the sheaf of holomorphic functions.
\item[(2)] {\it The non-Archimedean case}. As we mentioned in the introduction, an interesting choice of field is $K=\mathbb{Q}_p^{\text{alg}}$; and an interesting choice of topological space is the Berkovich space $\mathbb{A}_{\mathbb{C}_p}^{m,\text{an}}$ associated to the completion $\mathbb{C}_p$ of $\mathbb{Q}_p^{\text{alg}}$.
\end{enumerate}

\noindent
{\it Problem 3.} Study the analytic properties of lifts of tropical DA varieties.

\subsection{Seminorms, bis}
We saw in Proposition~\ref{Longue_prop} that the support map $\text{sp}:{K}_{m,n}\rightarrow \mathbb{B}_{m,n}$ is a $K$-algebra norm and also interacts in an interesting way with differentials. %Concretely, it is a non-degenerate ${K}$-algebra quasivaluation for which $\text{Supp}\circ\Theta_{{K}_{m,n}}(e_i)\leq \Theta_{\mathbb{B}_{m,n}}(e_i)\circ \text{Supp}$ for $i=1,\ldots,m$. 
In particular, it may be useful in the future study of tropical differential algebra from the perspective of differential (idempotent) semirings.

%\medskip
\noindent {\it Question 1.} Is there a well-behaved notion of tropical differential ideal?

%\medskip
 The fundamental theorem characterizes the systems of tropical differential equations $\Sigma\subset \mathbb{B}_{m,n}$ whose solution sets   equal the tropicalizations of differential algebraic sets $\text{Sol}(G)$ over an algebraically closed uncountable field of characteristic zero. 

%\medskip
\noindent {\it Question 2.} Is there an axiomatic characterization of such tropical systems? 

%\medskip
Classically, a tropical prevariety is a finite intersection of tropical hypersurfaces; see \cite[p.102]{MS}. %We have seen in Example \ref{ex_trop_hyp} that the concept of tropical DA hypersurface exists.
On the other hand, the concept of tropical DA hypersurface is clearly well-defined (and is illustrated explicitly in Examples~\ref{ex_trop_hyp} and \ref{ex_trop_hyp_2}).

%\medskip
\noindent {\it Question 3.} Is there a well-defined notion of tropical DA prevariety?

The theory of initial forms with respect to the non-Krull valuation trop introduced in section~\ref{Section_Initial_forms} also opens the door to a number of natural questions. %opens the path to the following type of questions:

%\medskip
\noindent {\it Question 4.} Is there a theory of Groebner bases adapted to the valuations trop: $K[\![T]\!] \rightarrow V\mathbb{B}[T]$ and trop:$K(\!(T)\!) \rightarrow V\mathbb{B}(T)$ when $m>1$?

%\medskip
In \cite{HG19} such a Groebner theory is developed when $m=1$, in which case trop is a Krull valuation.

In section~\ref{Section_Fromringtofield}, we extended the trop valuation from the domain $K[\![T]\!] $ to its quotient field $K(\!(T)\!)$. In doing so, we established in Proposition \ref{ring_of_integers} that whenever $m>1$ the ring of integers $K(\!(T)\!)^{\circ}$ is not a local ring, and that it properly contains $K[\![T]\!]$. Further, according to Lemma \ref{toric_degeneration}, %that for every $w\in\mathbb{B}[\![T]\!] ^n$, 
there is a translation map $P\mapsto P_w$ associated to every given weight vector $w\in\mathbb{B}_m^n$ that sends a differential polynomial with coefficients in $K(\!(T)\!)$ to a differential polynomial with coefficients in $K(\!(T)\!)^{\circ}$. We correspondingly considered the spectra $\mathcal{X}(w):=\text{Spec}(R_w)$ for $R_w$ as in \eqref{ring_for_model}. Each $R_w$ potentially represents an {\it integral model} of $R$; it remains, however to check that it is flat.
 
%\medskip
 \noindent {\it Question 5.}
Given $w$, is $\mathcal{X}(w)$ {\it flat} over the ring of integers $K(\!(T)\!)^\circ$?
%Study the structure of $K(\!(T)\!) ^{\circ}$ when $m>1$. In particular, can we describe Spec$(K(\!(T)\!)^{\circ})$? 

%%\medskip
%  From this perspective, initial degeneration also involves choosing a point $x\in \text{Spec}(K(\!(T)\!)^{\circ})$ and taking the fiber $\mathcal{X}(w)_x$ of our model over it, thus a subsidiary question would be to describe this scheme. 

%\medskip
In forthcoming work \cite{FGH21} we will show that the answer is ``yes". A subsidiary question with obvious geometric implications would be to study the scheme $\text{Spec}(K(\!(T)\!)^{\circ})$ that parameterizes our initial degenerations. 

   Finally, in order to prove the missing item (1) of the fundamental theorem \ref{EFT} for $F_{m,n}$ when $m>1$, we require a differential enhancement in the sense of Remark \ref{rem_diff_enh} of the basic tropicalization scheme $\trop:K(\!(T)\!) \to V\mathbb{B}(T)$. See  \cite[Ex. 7.4]{FGH20} for a related discussion when $m=1$.
 %\end{ques}
 
 %\medskip
\noindent {\it Question 6.} For every positive integer $m \geq 1$, is there an analogue of item (1) of the fundamental theorem \ref{EFT} for the field $F_{m,n}$?

\subsection{An example}
 \label{Example}

To  illustrate how tropical differential algebra works, in this subsection we compute $P(a)$ for particular choices of $P\in \mathbb{B}[\![t,u]\!]\{x,y\}$ and $a\in \mathbb{B}_2=\mathbb{B}[\![t,u]\!]$; and we give a general strategy for computing the solution set of an arbitrary differential polynomial $P\in\mathbb{B}_{m,n}$. Accordingly, let $a=(a_1,a_2)=(t+t^2+tu+u^3,1+u+t^2u+t^3u^2)$ and let %$P$ be the following trinomial of order one
%an example of a polynomial of order at most 1 would be \[
\begin{equation}
\label{concrete_example}
    P=(t+u)x_{(0,0)}y_{(1,1)}^3+(1+t^2u^2)x_{(1,0)}x_{(0,1)}+ty_{(1,0)}^2.
\end{equation}
Since there are comparatively few variables and the differential operators  involved are of low order, we will use standard partial derivative notation. We then have
\begin{itemize}
 \item $a_{M_1}E_{M_1}(a)=(t+u)a_1(\tfrac{\partial^2 a_2}{\partial t\partial u})^{3}=(t+u)(t+t^2+tu+u^3)(t+t^2u)^{3}$;
 \item $a_{M_2}E_{M_2}(a)=(1+t^2u^2)(\tfrac{\partial a_1}{\partial t })(\tfrac{\partial a_1}{\partial u})= (1+t^2u^2)(1+u+t)(t+u^2)$; and
 \item $a_{M_3}E_{M_3}(a)=t(\tfrac{\partial a_2}{\partial t})^2 =t(tu+t^2u^2)^2=t^3u^2(1+tu+t^2u^2)$.
 \end{itemize}
 
 The evaluation $P(a)$ is the union of three formal series $a_{M_i}E_{M_i}(a)$, $i=1,2,3$; they are displayed in Figure~\ref{Figura}. 
\begin{figure}[!htb]
    \centering
    \includegraphics[scale=0.65]{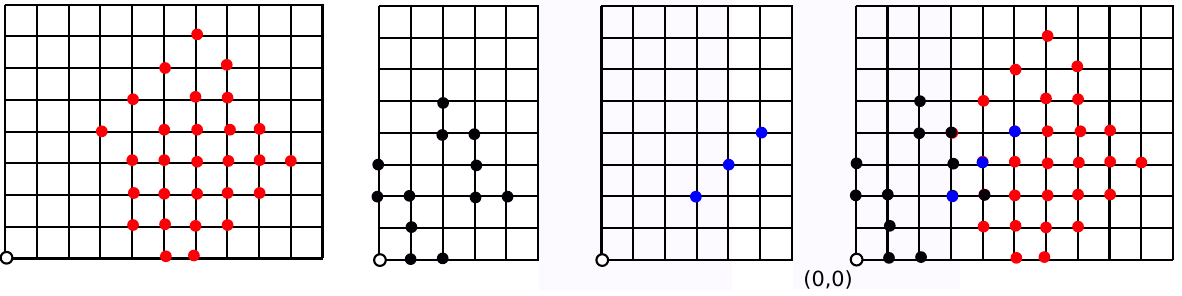}
    \caption{Values of $a_{M_i}E_{M_i}(a)$ for $i=1,2,3$, and $P(a)$.}
\label{Figura}
\end{figure}

\noindent Finally, note that $V(P(a))=V(a_{M_2})\odot V(E_{M_2}(a))=t+u^2$, and since %these 
$t$ and $u^2$ do not belong to the supports of either of the other two $a_{M_i}E_{M_i}(a)$, it follows that $a$ is not a solution of $P$.

%For the last part, we show how to work with non-archimedean amoebas.
We close by showing how to work with non-archimedean differential amoebae.

 \begin{rem}
 \label{non_arch_amoeba}
 The isomorphism of differential semirings $\text{Supp}:(\mathbb{B}[\![T]\!] ,D)\xrightarrow[]{}(\mathcal{P}(\mathbb{N}^m),D)$ from Remark~\ref{rem_difference} corresponds to 
 the map $\ell:{\rm S}\to {\rm S}^{\text{op}}$. As $\text{sp}:K_m \rightarrow\mathbb{B}_m$ is a seminorm, it follows that the amoeba of a DA variety $\text{Sol}(\Sigma)\subset K_m ^n$ is precisely its image under $\ell\circ\text{sp}:K_m ^n\rightarrow (\mathcal{P}(\mathbb{N}^m))^n$.
 
 %So, since $\text{sp}:K_m \rightarrow\mathbb{B}_m$ is a  seminorm, given a DA variety $\text{Sol}(\Sigma)\subset K_m ^n$, its  amoeba is its image under the map $\ell\circ\text{sp}:K_m ^n\rightarrow (\mathcal{P}(\mathbb{N}^m))^n$.
 \end{rem}

%We consider $\mathbb{N}^m$ endowed with the usual product order. 
Now equip $\mathbb{N}^m$ with the usual product order. Given $X\subset \mathbb{N}^m$ and $J\in \mathbb{N}^m$, let $X({\geq J}):=\{I\in X:\:I\geq J\}$. Note that 
\begin{equation*}
\label{concrete_description_sup_der}
    (X-J)_{\geq0}=\begin{cases}
    X({\geq J})-J,&\text{ if }X({\geq J})\neq\emptyset; \text{ and}\\
    \emptyset,&\text{ otherwise}.\\
    \end{cases}
\end{equation*}

Any differential monomial $aE$, with $E=\prod_{1\leq i\leq n, J\in\mathbb{N}^m}x_{i,J}^{m_{i,J}}$ and $A=\text{Supp}(a)$, induces an operator $aE:\mathcal{P}(\mathbb{N}^m)^n\xrightarrow[]{}\mathcal{P}(\mathbb{N}^m)$ defined by 
\begin{equation*}
%\label{explicit_computation}
   aE(X_1,\ldots,X_n)= \begin{cases}
   \emptyset,\text{ if }X_i(\geq J)=\emptyset\text{ for some }(i,J)\text{ in }E;&\text{ and}\\
   [\sum_{i,J}m_{i,J}X_{i}(\geq J)-\sum_{i,J}m_{i,J}J]+\text{Vert}(A)&\text{ otherwise.}\\
    \end{cases}
\end{equation*}
%{\color{red}
We may rewrite the second expression as $\sum_{i,J}m_{i,J}X_{i}(\geq J)+[\text{Vert}(A)-\sum_{i,J}m_{i,J}J]$ provided we authorize vectors with entries in $\mathbb{Z}^n$. %We illustrate this procedure with the polynomial $P$ from  
Executing this procedure using the polynomial $P$ from \eqref{concrete_example}, the maps $\mathcal{P}(\mathbb{N}^2)^2 \xrightarrow[]{}\mathcal{P}(\mathbb{N}^2)$ induced by its monomials are \begin{equation*}
    \begin{aligned}
    (t+u)x_{(0,0)}y_{(1,1)}^3(X,Y)&=X+3Y({\geq(1,1)})+\{(-2,-3),(-3,-2)\};\\
    (1+t^2u^2)x_{(1,0)}x_{(0,1)}(X,Y)&=X({\geq(1,0)})+X({\geq(0,1)})+(-1,-1); \text{ and}\\
    ty_{(1,0)}^2(X,Y)&=2Y({\geq(1,0)})+(-1,0).
    \end{aligned}
\end{equation*}

These equations impose restrictions on the support vector ($X,Y)\in\mathcal{P}(\mathbb{N}^2)^2$ of solutions of $P$. It remains to carry out a case-by-case analysis of the expression $V(P(a))=\bigoplus_MV(a_ME_M(a))$ based on the number of relevant summands, which must be at least two in order for $a$ to be a solution of $P$.   The casewise inspection required here is very similar to that required to explicitly describe a tropical hypersurface $V(f)$ as a union of semialgebraic sets in tropical affine space.
\appendix\label{Proof_of_Longue_prop}
\section{Proof of Proposition \ref{Longue_prop}}
Given integers $m>0$ and $n \geq 0$, we will show that $\text{sp}:{K}_{m,n}\rightarrow \mathbb{B}_{m,n}$ is a  ${K}$-algebra norm for which $\text{sp}\circ\Theta_{{K}_{m,n}}(e_i)\leq \Theta_{\mathbb{B}_{m,n}}(e_i)\circ \text{sp}$ for every $i=1,\ldots,m$.
 
\begin{proof}
The facts that $\text{sp}(P)=0$ if and only if $P=0$, that $\text{sp}(a)=1$ for every nonzero element $a\in {K}$, and that $\text{sp}\circ\Theta_{{K}_{m}}(J)(a)= \Theta_{\mathbb{B}_{m}}(J)\circ \text{sp}(a)$ for every $a\in {K}_m$ and $J\in\mathbb{N}^m$ are clear. 

%\medskip
Suppose now that $n=0$. Given elements $\alpha,\beta\in {K}[\![T]\!] $ with $\alpha=\sum_{I\in A}a_IT^I$ and $\beta=\sum_{I\in B}b_IT^I$, let $\text{sp}(\alpha)=\sum_{I\in A}T^I$ and $\text{sp}(\beta)=\sum_{I\in B}T^I$ in $\mathbb{B}[\![T]\!]$ denote their corresponding support series. We define
the {\it additive and multiplicative supports} %$\text{as}(a,b),\text{ms}(a,b)\in \mathbb{B}[\![T]\!] $ of $(a,b)$, respectively, by 
$\text{as}(\alpha,\beta),\text{ms}(\alpha,\beta)$ of $(\alpha,\beta)$, respectively, by 
\begin{equation}
\label{tad_and_tmd}
\text{as}(\alpha,\beta):=\sum_{\substack{I\in A\cup B\\a_I+b_I=0}}T^I \text{ and } \text{ms}(\alpha,\beta):=\sum_{\substack{I\in AB\\\sum_{I=K+L}a_Kb_L=0}}T^I.
\end{equation}

\noindent These are the unique elements of $\mathbb{B}[\![T]\!] $ for which
\begin{enumerate}
\item $\text{sp}(\alpha +\beta)\cap \text{as}(\alpha,\beta)=\emptyset$ and $\text{sp}(\alpha +\beta)+\text{as}(\alpha,\beta)=\text{sp}(\alpha) +\text{sp}(\beta)$;  
\item $\text{sp}(\alpha \beta)\cap \text{ms}(\alpha,\beta)=\emptyset$ and $\text{sp}(\alpha\beta)+\text{ms}(\alpha,\beta)=\text{sp}(\alpha)\text{sp}(\beta).$
\end{enumerate}

\noindent The subadditive and submultiplicative properties of \text{sp} follow from the existence of %\text{as} and \text{ms}, 
additive and multiplicative supports, respectively; and the desired result when $n=0$ follows. %immediately. We call $\text{as}(a,b)$ the { additive support} and $\text{ms}(a,b)$ the {  multiplicative support} of $a$ and $b$, respectively. 

%\medskip
More generally now, given $P=\sum_{M\in\Lambda(P)}a_ME_M$ and $Q=\sum_{M\in\Lambda(Q)}b_ME_M$, we have 
$P+Q=\sum_{M\in \Lambda(P)\cup \Lambda(Q)}(a_M+b_M)E_M$, 
and the fact that $\text{sp}(a_M+b_M)+\text{as}(a_M,b_M)=\text{sp}(a_M)+\text{sp}(b_M)$ for every $M\in\Lambda(P)\cup \Lambda(Q)$ implies that $\text{sp}(P+Q)\leq\text{sp}(P)+\text{sp}(Q)$.

%\medskip
Similarly, to prove that $\text{sp}(PQ)\leq\text{sp}(P)\text{sp}(Q)$, our point of departure is the fact that $PQ=\sum_{O\in \Lambda(P)+ \Lambda(Q)}c_OE_O$, where $c_O=\sum_{M+N=O}a_Mb_N$. We need to show that for any $O\in\Lambda(P)+\Lambda(Q)$, there is some $\gamma_O\in \mathbb{B}[\![T]\!]$ for which $\text{sp}(c_O)+\gamma_O=\sum_{M+N=O}\text{sp}(a_M)\text{sp}(b_N)$. This, in turn, follows from the facts that
\begin{enumerate}
\item $\text{sp}(c_O)+\text{as}_{M+N=O}(a_Mb_N)=\sum_{M+N=O}\text{sp}(a_Mb_N)$ and
\item $\text{sp}(a_Mb_N)+\text{ms}(a_M,b_N)=\text{sp}(a_M)\text{sp}(b_N)$.
\end{enumerate}

\noindent
The final condition on derivatives follows from the following observation: given $P=\sum_M\alpha_ME_M$ and $i=1,\ldots,m$, we have $\Theta_{{K}_{m,n}}(e_i) P=\sum_M(\Theta_{{K}_{m,n}}(e_i)\alpha_M)E_M+\sum_M\alpha_M(\Theta_{{K}_{m,n}}(e_i)E_M)$
and therefore %by item (2) we have
\[
\text{sp}(\Theta_{{K}_{m,n}}(e_i) P)\leq \sum_M\text{sp}(\Theta_{{K}_{m,n}}(e_i)\alpha_M)E_M+\sum_M\text{sp}(\alpha_M)(\Theta_{{K}_{m,n}}(e_i)E_M).
\]
To conclude, we note that 
\[
(\Theta_{{K}_{m,n}}(e_i)E_M)=(\Theta_{\mathbb{B}_{m,n}}(e_i)E_M) \text{ and } \text{sp}\circ\Theta_{{K}_{m,n}}(e_i)(\alpha)= \Theta_{\mathbb{B}_{m,n}}(e_i)\circ \text{sp}(\alpha)
\]
for every $\alpha\in {K}_{m,0}$, as this holds when $n=0$. 
\end{proof}

\begin{coro}
 \label{coro_Longue_prop}
For every $w\in \mathbb{B}_m^n$, the  map
$\trop_w:K[\![T]\!] [x_{i,J}: i,J]\ra V\mathbb{B}[T]$ is a multiplicative $K$-algebra  seminorm.
\end{coro}
\begin{proof}
The assignment $P\mapsto \trop_w(P)$ factors as $$P\mapsto\text{sp}(P)\mapsto \text{sp}(P)(w)\mapsto V(\text{sp}(P)(w)).$$
Thus $\trop_w(P)=0$ whenever $\text{sp}(P)(w)=0$.

We now appeal to Proposition~\ref{Longue_prop}, the algebra homomorphism properties of the evaluation map, and the linearity of $V$, in that order.
Explicitly, given $P,Q\in K[\![T]\!] [x_{i,J}: i,J]$ there are some $R,R'\in \mathbb{B}[\![T]\!] [x_{i,J}: i,J]$ for which $\text{sp}(P+Q)+R=\text{sp}(P)+\text{sp}(Q)$ and $\text{sp}(PQ)+R'=\text{sp}(P)\text{sp}(Q)$. Evaluating at $w$ now yields $\text{sp}(P+Q)(w)+R(w)=\text{sp}(P)(w)+\text{sp}(Q)(w)$ and $\text{sp}(PQ)(w)+R'(w)=\text{sp}(P)(w)\text{sp}(Q)(w)$ in $\mathbb{B}[\![T]\!]$; finally, the linearity of $V$ implies that $\trop_w$ is a $K$-algebra seminorm.

%\medskip
It remains to show that $\trop_w$ is multiplicative. Indeed, given  $P\in K[\![T]\!] [x_{i,J}: i,J]$ and  $w\in\mathbb{B}_m ^n$, every element of $\{\text{sp}(P(\varphi))\::\varphi\in K_m ^n,\:\:\text{sp}(\varphi)=w\}$
 is contained in $\text{sp}(P)(w)$. %Moreover, the evaluation $P(\varphi)$ will attain this maximal support 
 Moreover, for a generic choice of $\varphi\in K_m ^n$ with support $w$, we have $\text{sp}(P)(w)=\text{sp}(P(\varphi))$, which implies that $\trop_w(P)=V(\text{sp}(P)(w))=V(\text{sp}(P(\varphi)))= \text{trop}(P(\varphi))$.
 
 %%\medskip
\noindent Now assume $P_3=P_1P_2$, and choose   $\varphi\in K_m ^n$ for which $\text{sp}(\varphi)=w$ and $\text{sp}(P_i(\varphi))=\text{sp}(P_i)(w)$ for $i=1,2,3$.
 %Then using that the evaluation is multiplicative and the multiplicativity of trop, we get.  
As evaluation and trop are multiplicative, we have 
\[
\trop_w(P_1P_2)=\text{trop}(P_3(\varphi))=\trop(P_1(\varphi))\odot \trop(P_2(\varphi))=\trop_w(P_1)\odot \trop_w(P_2)
\]
which finishes the proof.
 \end{proof}

\textbf{Acknowledgements}. %The authors thank Lara Bossinger, Pedro Luis del \'Angel and Jeffrey Giansiracusa for interesting conversations and valuable comments; Alicia Dickenstein, for detailed comments on the first version of this paper; and Yue Ren, from whom we first learned of the potential applications of these methods to $p$-adic differential equations during a BIRS talk he gave in June 2020.
The authors thank Lara Bossinger, Alicia Dickenstein, Jeffrey Giansiracusa, and Pedro Luis del \'Angel for interesting conversations and valuable comments on previous versions of this paper; and Yue Ren, from whom we first learned of the potential applications of these methods to $p$-adic differential equations during a BIRS talk he gave in June 2020.

\end{document}